\documentclass[12pt]{amsart}
\usepackage{dsfont}
\usepackage{enumerate}
\usepackage{amssymb}
\usepackage{amsmath}
\usepackage[all]{xy}
\usepackage{tikz}
\usepackage{lmodern,bm}

\usepackage{pgf}
\def\obrazek#1#2{\pgfdeclareimage[height=#2]{#1}{#1}
\pgfuseimage{#1}}

\def\Ne{{\mathcal N}}
\def\Q{\mathbb Q}
\def\T{\mathbb T}
\def\P{\mathbb P}
\def\hbar{h}
\def\quapr{\ }

\DeclareMathOperator\R{\mathbb R}
\DeclareMathOperator\C{\mathbb C}
\DeclareMathOperator\Z{\mathbb Z}
\DeclareMathOperator\N{\mathbb N}
\DeclareMathOperator\I{\mathcal I}
\DeclareMathOperator\Sym{Sym}
\DeclareMathOperator\csm{c^{sm}}
\DeclareMathOperator\spa{span}
\DeclareMathOperator\rank{rank}

\DeclareMathOperator\Hom{\mathrm Hom}
\DeclareMathOperator\Gr{\mathrm Gr}
\DeclareMathOperator\codim{codim}
\DeclareMathOperator\ch{conv}
\DeclareMathOperator\Var{Var}
\DeclareMathOperator\id{id}
\DeclareMathOperator\pt{pt}
\DeclareMathOperator\Rep{R}
\DeclareMathOperator\rk{rk}
\DeclareMathOperator{\GL}{GL}

\DeclareMathOperator\Fl{Fl}
\DeclareMathOperator\F{\mathcal F}

\newcommand{\aalpha}{{\boldsymbol \alpha}}
\newcommand{\bbeta}{{\boldsymbol \beta}}

\def\G{{\bf G}}
\def\HH{{\bf H}}
\def\Tc{{\bf T}}

\newtheorem{fact}{Fact}[section]
\newtheorem{lemma}[fact]{Lemma}
\newtheorem{theorem}[fact]{Theorem}
\newtheorem{definition}[fact]{Definition}
\newtheorem{example}[fact]{Example}
\newtheorem{rremark}[fact]{Remark}
\newenvironment{remark}{\begin{rremark} \rm}{\end{rremark}}
\newtheorem{proposition}[fact]{Proposition}
\newtheorem{corollary}[fact]{Corollary}

\DeclareMathOperator\tauy{mC}
\DeclareMathOperator\ts{mS}%
\def\tilde#1{\widetilde{#1}}

\author{L\'aszl\'o M. Feh\'er}
\address{Institute of Mathematics, E\"otv\"os University Budapest, Hungary}
\email{lfeher63@gmail.com}

\author{Rich\'ard Rim\'anyi}
\address{Department of Mathematics, University of North Carolina at Chapel Hill, USA}
\email{rimanyi@email.unc.edu}

\author{Andrzej Weber}
\address{Institute of Mathematics, University of Warsaw, Poland}
\email{aweber@mimuw.edu.pl}

\title{Motivic Chern classes and K-theoretic stable envelopes}

\begin{document}

\begin{abstract}
We consider a smooth algebraic variety with an action of a linear algebraic
group acting with finitely many orbits. We study equivariant
characteristic classes of the orbits, namely the equivariant {\em motivic Chern classes}, in the K-theory of the ambient space.
We prove that the motivic Chern class satisfies the axiom system inspired by that of ``K-theoretic stable envelopes'', recently defined by Okounkov and studied in relation with quantum group actions on the K-theory algebra of moduli spaces. We also give explicit formulas for the equivariant motivic Chern classes of Schubert cells and matrix Schubert cells. Lastly, we calculate the equivariant motivic Chern class of the orbits of the $A_2$ quiver representation, which yields formulas for the motivic Chern classes of determinantal varieties and more general degeneracy loci.
\end{abstract}

\maketitle
\section{Introduction}

Characteristic classes of singular varieties are important tools in algebraic and enumerative geometry. These classes live in some extraordinary cohomology theory, and even within the same theory there are a few different flavors of them. For example, in ordinary cohomology theory one has the {\em Chern-Schwartz-MacPherson} (CSM) class (together with its variants, the {\em Segre-Schwartz-MacPherson} (SSM) class, and the {\em characteristic cycle}), the {\em Fulton-Johnson} class, and the {\em Chern-Mather} class.
In K-theory one has Brasselet-Sch\"urmann-Yokura's {\em motivic Chern class} (or its cohomology shadow, the {\em Hirzebruch class}). The corresponding $\chi_y$-genus is a common generalization of the L-genus and the Todd-genus.

In the present paper we define and study the {\em equivariant motivic Chern class}. After their definition in Section 2 we give three contributions to their theory and computation, detailed in the next three subsections. The first contribution is of general nature, while the following two are illustrations of the general principle in the particular cases, intensively studied by many authors.

\subsection{K-theoretic stable envelopes, motivic Chern classes}

We consider a smooth algebraic variety with an action of an algebraic group acting with finitely many orbits. We study equivariant motivic Chern classes of the orbits in the K-theory of the ambient space. In Theorems \ref{thm:intchar} and \ref{thm:axiomatic} we present an axiomatic characterization of the equivariant motivic Chern class. The axiom system is inspired by works of Okounkov. The Bia{\l}ynicki-Birula decomposition plays a crucial role in the proof of one of our axioms (the Newton-polytope axiom), showing that the asymptotic behavior of motivic Chern classes is related with the geometry of attracting sets with respect to one parameter subgroups of the torus.

There are situations when both Okounkov's {\em K theoretic stable envelopes} and our equivariant motivic Chern classes are defined, and after some identifications they live in the same ring. Then the analogy between their axiom systems imply that (under some hypotheses, see details in Remark~\ref{rem:comparison}) the notion of motivic Chern class coincides with the notion of stable envelope for a particular choice of alcove parameter.

 In \cite{RV2} (see also \cite{AMSS}) analogous results are proved in cohomology instead of K-theory: the Chern-Schwartz-MacPherson classes of certain singular varieties are characterized by axioms and they are related with {\em cohomological} stable envelopes.



\subsection{Motivic Chern classes of matrix Schubert varieties}

The theory of {\em cohomological fundamental classes} is well developed. For quiver representations the fundamental class is called quiver polynomial, for singularities of maps it is called Thom polynomial. In any of these areas and in others, it turns out that cohomological fundamental classes {\em should be} expressed in Schur polynomials---because those expansions show stability and positivity properties, e.g. \cite{PW, BR}. Moreover, Schur polynomials are cohomological fundamental classes themselves, namely those of the so-called matrix Schubert varieties \cite{FRSchur}.

The theory of {\em K-theoretic fundamental classes} has a similar shape. The atoms of the theory are Grothendieck polynomials: K-theoretic fundamental classes in quiver and Thom polynomial settings show stability and positivity properties as soon as they are expanded in (stable double) Grothendieck polynomials \cite{B, RS}. Moreover, these Grothendieck polynomials are the K-theoretic fundamental classes of matrix Schubert varieties.

The theory of {\em CSM/SSM classes} (i.e. the theory of cohomological stable envelopes) also has a similar shape. The
atoms of the theory are the $\tilde{s}_\lambda$ classes of \cite{FR}: SSM classes of geometrically relevant varieties, when expanded in the $\tilde{s}_\lambda$ classes show stabilization and positivity properties. Moreover, $\tilde{s}_\lambda$ classes are the SSM classes of matrix Schubert cells.

Hence, it is natural to predict that the atoms of the theory of motivic Chern classes all over geometry will be the motivic Chern classes of matrix Schubert cells. In Theorems \ref{thm:cptmCW} and \ref{thm:tauyMO} we present formulas for the equivariant motivic Chern classes of Schubert cells in partial flag varieties and those of matrix Schubert cells. The formulas will be appropriate versions of the so-called trigonometric weight functions of Tarasov-Varchenko, whose role in Schubert calculus was first discovered in \cite{RTV1}.

\subsection{Determinantal varieties}
In Section \ref{sec:A2} we study the motivic Chern classes of the orbits of $\Hom(\C^k,\C^n)$ acted upon by $\GL_k(\C)\times \GL_n(\C)$. This representation is called the $A_2$ quiver representation, and after projectivization the orbit closures are called determinantal varieties. We prove formulas for the motivic Chern classes of the orbits, both in the traditional ``localization'' form, and also expanded in the building blocks invented in Section \ref{sec:tauyMO}. These results can be viewed as the K-theory generalizations of \cite{PP, FR, xiping}.

\smallskip

In the paper there will be three concrete calculations of equivariant motivic Chern classes: those of Schubert cells, of matrix Schubert cells, and of $A_2$ orbits. These three calculations illustrate the three different methods known to the authors: calculation by interpolation axioms, by resolution with normal crossings, and by the sieve of general resolutions.

\subsection{Acknowledgments}
We are grateful to J. Sch\"urmann for teaching us the notion of motivic Chern class, and for his very helpful comments that significantly improved our paper. We also thank A. Varchenko, S. Jackowski, J. Wi{\'sniewski} for discussions on related topics, and B.~Maultsby for help with visualizations.
L.F.  is  supported by  NKFIH 126683, NKFI 112703 and 112735. R.R. is supported by the Simons Foundation grant 523882. A.W. is supported by the research project of the Polish National Research Center 2016/23/G/ST1/04282 (Beethoven 2, German-Polish joint project).

\section{The motivic Chern class}

In this section we recall the notion of motivic Chern class from \cite{BSY}, see also \cite{CMSS,Sch-MHM}.
The images of motivic Chern classes in cohomology are called Hirzebruch classes. We will set up the equivariant versions of these classes. The equivariant versions of Hirzebruch classes were studied in \cite{WeSelecta,WeBB}.

\subsection{Defining invariants of singular varieties}
Many invariants of smooth compact algebraic varieties can be defined through some operation on the tangent bundle. It is natural to ask whether they can be extended to singular varieties, which of course do not have tangent bundles. The same question applies to open manifolds which do not possess a fundamental class, but can be compactified. One natural approach is to resolve singularities or compactify and apply the preferred invariant to the resolution or compactification. In general the result depends on the resolution/compactification. One can try to modify the result by correction terms coming from the exceptional or boundary divisors. This way Batyrev defined stringy Hodge numbers
\cite{Ba} for varieties with at worst Kawamata log-terminal singularities.
Borisov-Libgober \cite{BoLi} have defined the elliptic cohomology class for the same class of singularities. The construction extends to pairs, also with a group action.
Taking a resolution one usually assumes that the exceptional locus is a sum of smooth divisors intersecting transversely. To prove that the invariant does not depend on the resolution
it is enough to check that it does not change when we modify the resolution by blowing up a smooth center well posed with respect to the existing configuration of divisors.
This is a consequence of Weak Factorization Theorem \cite{Wlodek}. The construction can be carried out in the presence of a group action since by Bierstone and Milman \cite{BiMi} equivariant resolutions exist and the Weak Factorization can be realized in the invariant manner.
For example the Chern-Schwartz-MacPherson class can be defined using this technique, as discovered by Aluffi.  Moreover this way Aluffi obtained new results, similar to the invariance of Batyrev's stringy Hodge numbers.
\medskip

We will apply the described method {\em to define} the equivariant motivic Chern class of singular varieties with arbitrary singularities. In later sections we will present three approaches to their calculation: the direct application of the definition,
 see Theorem \ref{thm:tauyMO}, an axiomatic characterization, see Theorems \ref{thm:intchar}, \ref{thm:axiomatic} and Theorem \ref{thm:cptmCW}, and the sieve method illustrated in Section~\ref{sec:sieve}.

\subsection{The $\tauy$-classes}
Consider the two covariant functors $\Var(-)$ and $K_{alg}(-)[y]$
\[
(\text{smooth  varieties, proper maps})\longrightarrow
(\text{groups, group homomorphisms}).
\]
We do not have to assume that the varieties are quasi-projective.
Recall that $\Var(M)$ is generated by classes of maps $X\to M$, where $X$ can be singular, and $f$ is not necessarily proper, modulo additivity relations see \cite{BSY}. Recall that $K_{alg}(M)[y]$ is the K-theory ring of $M$ with $y$ a formal variable. By K-theory we mean the Grothendieck group generated by locally free sheaves. See the Appendix for a review of algebraic K-theory. The version of K-theory built from the coherent sheaves is denoted by $G(M)$. If $M$ is smooth, then the natural map $K_{alg}(M)\to G(M)$ is an isomorphism. We should have in mind that the push-forward of a sheaf is in a natural way represented in $G(M)$ by a coherent sheaf, which is not necessarily locally free. In $K_{alg}(M)$ the class of this sheaf is represented by its resolution.  The topological K-theory $K_{top}(M)$ is the Grothendieck group made from topological vector bundles. The natural map $K_{alg}(M)\to K_{top}(M)$ in general is not an isomorphism except some rare cases, for example when the space admits a decomposition into algebraic cells. We show that both theories are isomorphic under the assumptions of Theorem \ref{thm:axiomatic}, that is in the situation of our main interest.

The motivic Chern class $\tauy$ is a natural transformation
\[
\Var(M)\to K_{alg}(M)[y].
\]
Thus for a map $X\to M$ we have
 $\tauy([X\to M]) \in K_{alg}(M)[y]$. When there is no confusion we will drop the brackets from the notation and write only $\tauy(X\to M)$, or the even simpler $\tauy(X)$ if $X\subset M$. The transformation of functors $\tauy$ is the unique {\em additive} transformation satisfying the normalization condition
\begin{equation}\label{eqn:normalize}
\tauy(\id_M)=\lambda_y(T^*M),
\end{equation}
where for a vector bundle $E$ we define $\lambda_y(E):=\sum_{i=0}^{\rank E}[\Lambda^iE]y^i$.

In \cite{BSY} singular spaces $M$ were also allowed at the cost of replacing $K_{alg}(M)$ with the K-theory of coherent sheaves $G(M)$. We will not need that generality in this paper, and hence we will not use the notation $G(M)$ any more.

\subsection{The equivariant $\tauy$ class} Let $G$ be an algebraic linear group.
We will use the $G$-equivariant version of $\tauy$: for a smooth \quapr $G$-space $M$ and a $G$-equivariant map $f:X\to M$ of $G$-varieties we assign an element $\tauy(f)=\tauy(X\to M)$ in $K_{alg}^G(M)[y]$. This assignment is additive and satisfies the normalization property (\ref{eqn:normalize}) now read equivariantly. The class satisfies the following properties:
\begin{description}
\item[1. Additivity] If $X=Y\sqcup U$, then $$\tauy(X\to M)=\tauy(Y\to M)+\tauy(U\to M)\,.$$

\item[2. Functoriality] For a proper map $f:M\to M'$ we have $$\tauy(X\stackrel{f\circ g}\to M')=f_*\tauy(X\stackrel{g}\to M)\,.$$

\item[3. Localness] If $U\subset M$ is an open set, then $$\tauy(X\stackrel{f}{\to} M)|U=\tauy(f^{-1}(U)\to U)\,.$$

\item[4. Normalization] We have $$\tauy(\id_M)=\lambda_y(T^*M)\,.$$

\end{description}
Moreover it is uniquely determined by the properties (1),(2) and (4).

By $K_{alg}^G(M)[y]$ we mean the equivariant K-theory built from $G$-equivariant locally free sheaves on $M$.
Since M is
smooth the map from equivariant K-theory of locally free sheaves to the
K-theory of equivariant coherent sheaves is an isomorphism, see
\cite[p.28]{Edi} and the Appendix.

We will apply the following construction of motivic Chern classes  $\tauy(X\to M)$ assuming that $X$ is smooth.

\subsection{Definition of $\tauy(X\to M)$ for smooth $X$}\label{tau-definicja}
 The usual way of defining  motivic Chern classes is using a special resolution.
 \begin{definition}\label{def:pnce} Suppose $f:X\to M$ is a map of smooth $G$-varieties. Then a \emph{proper normal crossing extension} of $f$ is a proper map $\bar f:Y\to M$ with an embedding $\iota:X\hookrightarrow Y$ satisfying $f=\bar f\circ \iota$, such that the variety $Y$ is smooth
 and the complement $Y\setminus \iota(X)=\bigcup_{i=1}^s D_i$
is a simple normal crossing divisor.
 \end{definition}

As discussed in \cite[\S5]{WeBB}
 such proper normal crossing extension always exists.
\medskip

 Now suppose that  $f:X\to M$ is a map of smooth \quapr $G$-varieties and let $\bar f:Y\to M$ a proper normal crossing extension of $f$.
 For $I\subset\underline{s}=\{1,2,\dots,s\}$ let $D_I=\bigcap_{i\in I}D_i$, $f_I=\bar f{|D_I}$, in particular $f_\emptyset=\bar f$.
Then
\begin{equation}\label{alternating-definition}\tauy(f)=\sum_{I\subset \underline{s}}\;(-1)^{|I|}f_{I*}\lambda_y(T^*D_I)\,.\end{equation}

 The independence of the chosen extension follows from a refined version of Weak Factorization, as formulated in \cite{Wlodek}. Here is a sketch of the arguments, essentially a repetition of the proof of \cite[Thm.~5.1]{Bittner}.
Suppose that $Y$ and $Y'$ are two proper normal crossing extensions of $X$.  By \cite[Thm.~0.0.1]{Wlodek} we can assume that $Y'$ is obtained by  a blowup of $Y$ in a smooth $G$-invariant center $C$ which has normal crossing\footnote{A smooth subvariety $C\subset Y$ has normal crossing with a divisor $D$ if at each point there exist a system of local coordinates such that the
components of the divisor containing this point are given by the vanishing
of one of the coordinates, and $C$ is given by the vanishing of a set of
coordinates.}  with $D=Y\setminus X$. Let  $\sigma:Y'\to Y$ be a blow-down. Let $E\subset Y'$ be the exceptional divisor and $D'_I$ be the proper transform of $D_I$.
Comparing the formulas corresponding to the two extensions we see that it remains to check that
\begin{equation}\label{eq:scisors}\sigma_*( \lambda_y(T^*D'_I))-\sigma_*( \lambda_y(T^*(D'_I\cap E)))=\lambda_y(T^*D_I)- i_*\lambda_y(T^*(D_I\cap C)\,,\end{equation}
where $i:D_I\cap C\to C$ is the inclusion. All the different restrictions of $\sigma$ are denoted by $\sigma$ to keep the notation short.

The equality means that for each integer $p$
$$\sigma_*( \Omega^p_{D'_I})-\sigma_*( \Omega^p_{D'_I\cap E})=\Omega^p_{D_I}- \Omega^p_{D_I\cap C}.$$
(We use the notation $\Omega^p$ for $\Lambda^p$ to emphasize that we are pushing forward sheaves.)
The right hand side has no higher cohomology. The equality in K-theory follows from the natural $G$-equivariant isomorphisms\begin{align*}
\Omega^p_{D_I}&\to R^0\sigma_*( \Omega^p_{D'_I}),\\
\Omega^p_{D_I\cap C}&\to R^0\sigma_*( \Omega^p_{D'_I\cap E})\quad\text{projective bundle},\\
R^k\sigma_*( \Omega^p_{D'_I})&\to R^k\sigma_*( \Omega^p_{D'_I\cap E})\quad \text{for }k>0,
\end{align*}
where $R^k\sigma_*$ is the $k$-th derived push forward.
These equivariant isomorphisms can be proved exactly like the corresponding nonequivariant ones in
\cite[Prop.~3.3]{GuNaAz}.

\begin{remark}\rm If $M$ is not smooth then we would have to use the K-theory of equivariant coherent sheaves, not necessarily locally free, as it is done in \cite[Cor.~0.1]{BSY} in the non-equivariant case.\end{remark}
\bigskip

For maps $X\to M$ with $X$ smooth the properties 2.~(Functoriality), 3.~(Localness) and 4.~(Normalization) follow directly from the definition given above. To show 1. (Additivity) we apply the extension of the Bittner \cite{Bittner} result in the equivariant setting: It is enough to check the following equality for a smooth and closed subvariety in a smooth ambient space $i:C\subset X$:
$$\sigma_*(\lambda_y(T^*Bl_CX))-\sigma_*(\lambda_y(T^*E))=\lambda_yT^*X-i_*(\lambda_*(T^*C))\,,$$
where $\sigma:Bl_CX\to X$ is the blow-up map. This formula is exactly a special case of the equality \eqref{eq:scisors}.
 To define and compute in practice the class $\tauy(X\to M)$ we decompose $X$ into the union of smooth locally closed subvarieties and sum up the $\tauy$--classes of the pieces. For details in the case when $G$ is a torus we direct the reader to \cite[Thm. 4.2]{AMSS2}. In all the sections except \S\ref{sec:A2} we will use only the properties 2.--4.~for smooth locally closed subvarieties of a smooth \quapr ambient space.
\bigskip

\noindent{\bf Notation:} In the  following sections except the Appendix by the $K^G(M)$ we will mean the algebraic K-theory $K^G_{alg}(M)$. In the Appendix the topological K-theory is discussed. We show that under the assumptions of Theorem \ref{thm:axiomatic} both theories coincide.

\begin{remark}\rm To construct equivariant Hirzebruch class in cohomology, which is the image of the motivic Chern class with respect to the Riemann-Roch transformation, one can  repeat with some care the arguments of \cite{BSY} line by line in the equivariant setting. This is done via approximating classifying spaces with their finite dimensional skeleta, analogously to how Ohmoto defined the equivariant version of Chern-Schwartz-MacPherson classes in \cite{O1,O2}. This program was carried out in \cite{WeSelecta,WeBB}. To define motivic Chern classes in K-theory we cannot follow this procedure. We have to construct directly an element in equivariant K-theory represented by an equivariant sheaf. We do not use here classifying spaces. Thus the obtained invariant is more refined.
\end{remark}

\subsection{The $\chi_y$ and equivariant $\chi_y$ genera}
The $\tauy$ class of a map $X\to \pt$ to a point is an element of the representation ring
$$\chi_y^G(X):=\tauy(X\to \pt)\in K^G(\pt)[y]=R(G)[y],$$
and it is called the (equivariant) $\chi_y$-genus of $X$.
If $G=\T$ is a torus, i.e. a product of $\GL_1(\C)$'s, then  $\chi_y^\T(X)$ contains no information about the action of $G$:
$$\chi_y^\T(X)=\iota (\chi_y(X))\,,\quad\text{ where } \quad\iota:\Z[y]=R(\{1\})[y]\to R(\T)[y]\,.$$
This property is called the rigidity of the $\chi_y$-genus. For the smooth case see \cite{Musin}.
Applying the equivariant Riemann-Roch transformations \cite{EdGr}
$$td_*^\T:K^\T(X)\to \hat{CH}^\T_*(X)\otimes \Q$$
to the completed equivariant Chow groups (or going from there then to the completed equivariant Borel-Moore homology) and extending by the formal parameter $y$, one obtains the equivariant Hirzebruch class $td_y^\T(X):=td^\T_*(\tauy(id_X))$ from \cite{WeSelecta}, with
$$td_*^\T:K^\T(\pt)\to \hat{CH}^\T_*(\pt)\otimes\Q\simeq \hat H^*_\T(\pt;\Q)$$
injective (as discussed in the appendix). Hence the cohomological rigidity for singular varieties from \cite[Thm.~7.2]{WeSelecta} implies the K-theoretical rigidity.
The same applies to connected (non necessarily reductive) groups for which $K^G(\pt)=R(\T)^W$, where  $\T$ is the maximal torus of $G$ and $W$ is the Weyl group.

\subsection{Notation conventions}

Let $\C^*$ act on $\C$ by $\xi\cdot x=\xi^k x$, and let $\ell$ be the $\C^*$-equivariant line bundle over the point with this action. We will use the notation
\[
K^{\C^*}(\C)=K^{\C^*}(\pt)=\Z[\xi^{\pm 1}], \qquad
[\ell]= \xi^k, \qquad
\text{and hence} \qquad
\lambda_y(\ell^*)=1+y/\xi^k.
\]
That is, by slight abuse of notation we use the same (Greek) letter for the coordinate of $\C^*$ and the generator of $K^{\C^*}(\pt)$.

The map $i^*:K^G(X)\to K^G(U)$ in $K$-theory induced by the embedding $i:U\subset X$ will be called a restriction map and will be denoted by $\alpha\mapsto \alpha{|U}$.

\subsection{The fundamental calculation}\label{sec:fundcalc}
The results of the following basic calculation will be essential in later sections. Let a torus $\T=(\C^*)^r$ act on $\C$, and let the class of this equivariant line bundle over a point be $\alpha\in K^{\T}(\pt)$. Using the notational conventions above, and the additivity of the motivic Chern class we have
\begin{equation}\label{eqn:key}
\tauy(\{0\}\subset \C)=1-1/\alpha,
\qquad
\tauy(\C\subset \C)=1+y/\alpha,
\end{equation}
\[
\tauy(\C-\{0\}\subset \C)=(1+y/\alpha)-(1-1/\alpha)=(1+y)/\alpha.
\]

\begin{remark} \label{rem:coh}
Those familiar with the Chern-Schwartz-MacPherson classes $\csm$ in equivariant cohomology may find it instructive to compare (\ref{eqn:key}) with the analogous expressions
\begin{equation*}\label{eqn:key2}
\csm(\{0\}\subset \C)=a,
\qquad
\csm(\C\subset \C)=1+a,
\qquad
\csm(\C-\{0\}\subset \C)=(1+a)-(a)=1.
\end{equation*}
The CSM class can be obtained from the formulas in (\ref{eqn:key}) by applying the following operations
\begin{itemize}
\item substitute $\alpha=\exp(at)$, $y=-\exp(-\hbar t)$;
\item take the coefficient of $t^d$ in the Taylor expansion in $t$, where $d$ is the dimension of the ambient space;
\item substitute $\hbar=1$;
\end{itemize}
while the class before substituting $\hbar=1$ is called the characteristic cycle class, or co\-ho\-mo\-lo\-gi\-cal stable envelope, see \cite{FR, AMSS} and references therein.
We note that the passage from Hirzebruch class to CSM class was given by Yokura, \cite[Rem.~2]{Yok} by certain substitution and specialization to $y=-1$.
The approach of \cite{BSY,Yok} allows different important specializations like $y=0$ or 1 and with $y=-1$ fitting with CSM classes. Our approach only works for the CSM classes, but gives a different view since it just corresponds to the equivariant Chern character (see also the Appendix, with $t$ fixing the (co)homological degree)
$$ch^{\T\times \C^*}:K^{\T\times \C^*}(\pt)\to \hat{CH}^\T_*(\pt)\otimes \Q[[h]]\simeq \hat H^*_\T(\pt;\Q)[[h]]\,,$$
with $-y=q^{-1}\in \Z[q^{\pm 1}]\simeq K^{\C^*}(\pt)$, see also \cite[\S7]{AMSS2}.

\end{remark}

\section{Newton polytopes and N-smallness}
In this section we consider the ring of Laurent polynomials in $r$ variables $\Z[\alpha_1^{\pm 1},\ldots,\alpha_r^{\pm1}]$.
This ring will enter our geometric calculations below as the $(\C^*)^r$-equivariant K-theory algebra of a point.

\subsection{Convex polytopes and their projections} \label{sec:conv}
A convex polytope in $\R^r$ is the convex hull of finitely many points in $\R^r$. For convex polytopes $U,V \subset \R^r$ we have the Minkowski sum operation
$
U+V=\{u+v\ |\ u\in U, v\in V\}.
$
It is known that this operation (on convex polytopes) satisfies the
{\em Cancellation Law}, \cite[Lemma1]{Ra}
\begin{equation}\label{eqn:CL}
U \subset W \ \Leftrightarrow \ U+V \subset W+V\qquad\text{for a fixed }V\,.
\end{equation}
For $s=(s_1,\ldots,s_r)\in \R^r$ define the linear map $\pi_s: \R^r \to \R$ by $\pi_s(u_1,\ldots,u_r)=s_1 u_1+s_2 u_2+\ldots+s_r u_r$.
The next proposition follows from the general theory of supporting hyperplanes of convex polytopes.

\begin{proposition}\label{prop:criterium}
Let $U$ and $V$ be convex polytopes in $\R^r$. Let $K$ be a union of finitely many hyperplanes in $\R^r$. The following conditions are equivalent:
\begin{itemize}
\item $U\subset V$;
\item $\pi_s(U)\subset \pi_s(V)$ for all $s \in \R^r$;
\item $\pi_s(U)\subset \pi_s(V)$ for all $s \in \Z^r - K$.
\end{itemize}
Moreover, if $U\subsetneq V$, then there exists an $s\in \Z^r - K$ such that $\pi_s(U)\subsetneq \pi_s(V)$. \qed
\end{proposition}

\subsection{Newton polytopes of Laurent polynomials and their toric substitutions}

For an element $f\in \Z[\alpha_1^{\pm 1},\ldots,\alpha_r^{\pm1}]$ we define its Newton polytope $\Ne(f)$ to be the convex hull
of the finite set
\[
\{(u_1,\ldots, u_r)\in \Z^r\ |\ \text{the coefficient of }
\alpha_1^{u_1}\ldots \alpha_r^{u_r}\text{ in } f \text{ is not } 0\}
\]
in $\R^r$. We have
\begin{equation} \label{eqn:Nprop}
\Ne(fg)=\Ne(f)+\Ne(g), \qquad
\Ne(f+g)\subset conv.hull(\Ne(f),\Ne(g)).
\end{equation}
\begin{definition}\label{def:toric}
For $s=(s_1,\ldots,s_r)\in \Z^r$ consider the map (called one parameter subgroup)\qquad
$
\kappa_s:\C^*\to (\C^*)^r,  \ \xi\mapsto (\xi^{s_1},\ldots,\xi^{s_r}).
$
The induced ring homomorphism
$
\Z[\alpha_1^{\pm1},\ldots,\alpha_r^{\pm1}]\to \Z[\xi^{\pm1}]
$
will be called the $\alpha=\xi^s$ toric substitution, and will be denoted by $f(\alpha)\mapsto f(\xi^s)$ or $f(\alpha)\mapsto f(\alpha)|_{\alpha=\xi^s}$.
\end{definition}

Now we discuss the relation between the projection $\pi_s$ of Section~\ref{sec:conv} and the toric substitution.

\begin{proposition}\label{prop:crit2}
Let $f,g\in \Z[\alpha_1^{\pm 1},\ldots,\alpha_r^{\pm 1}]$.
\begin{itemize}
\item
There exists a finite union of hyperplanes $K\subset \Z^r$ such that
\[
\Ne(f)\subset \Ne(g) \qquad \Longrightarrow \qquad \Ne(f(\xi^s))\subset \Ne(g(\xi^s)) \ \forall s\in \Z^r-K.
\]
\item If $K\subset \Z^r$ is a finite union of hyperplanes then
\[
\Ne(f(\xi^s))\subset \Ne(g(\xi^s)) \ \forall s\in \Z^r-K  \qquad \Longrightarrow \qquad \Ne(f)\subset \Ne(g).
\]
\item Let $K\subset \Z^r$ be a finite union of hyperplanes. If $\Ne(f)\subsetneq \Ne(g)$ then there is an $s\in \Z^r-K$ such that $\Ne(f(\xi^s))\subsetneq \Ne(g(\xi^s))$.
\end{itemize}
\end{proposition}

\begin{proof} The statements follow from Proposition \ref{prop:criterium} after observing that
\begin{equation}\label{eqn:comp}
\Ne( f( \xi^s) )
\begin{cases}
\subset \pi_s\left( \Ne( f) \right)
& \text{for all $s\in \Z^r$,} \\
 = \pi_s\left( \Ne( f) \right) & \text{for $s\in \Z^r-K,$}
\end{cases}
\end{equation}
where $K$ is a finite union of hyperplanes.
\end{proof}

\subsection{The N-smallness property}

\begin{definition}
Let a rational function $h(\alpha_1,\ldots,\alpha_r)$
be written as the ratio $f_1/f_2$, with $f_1,f_2\in \Z[\alpha_1^{\pm 1},\ldots,\alpha_r^{\pm1}]$. We say that $h$ is N-small (N for ``Newton polytope''), if $\Ne(f_1)\subset \Ne(f_2)$.
\end{definition}
Here $h=f_1/f_2$ belongs to the field of fractions $Frac(\Z[\alpha_i^{\pm1}])$ of the  integral domain $\Z[\alpha_i^{\pm1}]$.
Observe that the definition makes sense. Indeed, if $h$ is presented in two different ways as a ratio, then they are necessary $f_1/f_2=(f_1\cdot g)/(f_2\cdot g)$; and claims (\ref{eqn:CL}) and (\ref{eqn:Nprop}) imply that one presentation is N-small if and only if the other presentation is N-small.

\begin{lemma} \label{lem:addmult}
The set of N-small rational functions is closed under addition and multiplication.
\end{lemma}

\begin{proof}
Let $f_1/f_2$ and $g_1/g_2$ be N-small.
We have $\Ne(f_1g_2) = \Ne(f_1)+\Ne(g_2) \subset \Ne(f_2)+\Ne(g_2)=\Ne(f_2g_2)$. Similarly $\Ne(f_2g_1)
\subset
\Ne(f_2g_2)$. Then (\ref{eqn:Nprop}) implies $\Ne(f_1g_2 + f_2g_1) \subset \Ne(f_2g_2)$ which proves the claim for addition.
The claim for multiplication is proved by $\Ne(f_1g_1)=\Ne(f_1)+\Ne(g_1)\subset \Ne(f_2)+\Ne(g_2)=\Ne(f_2g_2)$.
\end{proof}

The following analytic characterization of N-smallness of rational functions will be useful.

\begin{proposition}\label{prop:analytic} Let $h(\alpha_1,\ldots,\alpha_r)$ be a rational function.
\begin{itemize}
\item There exists a finite union of hyperplanes $K\subset \Z^r$ such that
\[
h\text{ is N-small} \qquad\Longrightarrow\qquad \lim_{\xi\to \infty} h(\xi^s) \text{ is finite for all } s\in \Z^r-K.
\]
\item
Let $K\subset \Z^r$ be a finite union of hyperplanes. Then
\[
\lim_{\xi\to \infty} h(\xi^s) \text{ is finite for all } s\in \Z^r-K
\qquad\Longrightarrow\qquad
h\text{ is N-small}.
\]
\end{itemize}
\end{proposition}

\begin{proof} The statements follow from Proposition \ref{prop:crit2} and elementary calculus knowledge about the limits of rational functions in one variable.
\end{proof}

\section{Newton polytope properties of motivic Chern classes of subvarieties}
We recall the notion of the Bia{\l}ynicki-Birula cell which was originally defined for $\C^*$, but we need it for a one parameter subgroup of $\C^*\to\T$.

Let $\T=(\C^*)^r$ act on a variety $V$, and consider the one parameter subgroup $\kappa_s:\C^*\to \T$ of De\-fi\-ni\-tion~\ref{def:toric}. Let $F$ be a connected component of the fixed point set $V^{\kappa_s(\C^*)}$. The Bia{\l}ynicki-Birula cell is the subvariety
\[
V^{s}_F=\{x\in V\,:\, \lim_{\xi\to \infty} (\kappa_s(\xi)\cdot x)\in F\}
\,.\]
If $\T=\C^*$, $s=1$, then $\kappa_1=id$ and $V^{s}_F$ coincides with the minus-cell $V^-_F$ defined in \cite{BB}. If $0\in V$ is an isolated fixed point we will write $V^{s}_0$ instead of $V^{s}_{\{0\}}$. Note that the limit $\lim_{\xi\to \infty}$ may not exist if $V$ is not complete, and hence in that case the Bia{\l}ynicki-Birula cells may not cover the whole $V$. In our argument, for convenience we will apply the equivariant completion of \cite{Sum}. Then the space $V$ is decomposed into cells, see the discussion in \cite[Cor.~4]{WeBB}.

\begin{remark} If $V$ is a vector space with a linear action of $\T$, then the minus-cell $V^s_0\subset V$ of the origin $0\in V$ as an isolated fixed point is the subspace spanned by the weight vectors with negative weights with respect to the $\C^*$-action induced by~$\kappa_s$.
\end{remark}

The  Bia{\l}ynicki-Birula decomposition was studied recently by Drinfeld and Gaitsgory \cite{DG} with an application to derived categories. In their construction the functorial properties of the stable and unstable sets for a $\mathbb G_m$--action play a crucial  role, see also \cite{JeSi}. We need the functorial properties of the Bia{\l}ynicki-Birula decomposition to control the asymptotic behavior of the motivic Chern classes.

In Newton polytope considerations we will treat the variable $y$ as a constant, so the Newton polytope of $f=\sum a_I\alpha^I$ with $a_I\in \Z[y]$ is equal to $conv.hull\{I\;:\;a_I\neq 0\in\Z[y]\}$. Similarly the notion of $N$-smallness is defined for elements of $Frac(\Z[\alpha^{\pm 1}][y])$. Following \cite{WeBB} we have

\begin{theorem} \label{thm:smallness}
Let $V$ be a smooth \quapr variety with $\T=(\C^*)^r$ action. Suppose $0\in V$ is an isolated fixed point. Let $\Sigma \subset V$ be an invariant subvariety, not necessarily closed.
Then for almost all $s\in \Z^r$ (i.e. the set of exceptions is contained in a finite union of hyperplanes)
we have
\[
\lim_{\xi\to \infty}\left(\left.\frac{\tauy(\Sigma\subset V){|0}}{\lambda_{-1}(T^*_0V)}\right|_{\alpha=\xi^s}\right)=\;\chi_y(\Sigma\cap V^{s}_0)\,.
\]
Hence
$$\Ne( \tauy(\Sigma\subset V){|0}) \subset \Ne( \lambda_{-1}(T^*_0V))$$
and
$$\frac{\tauy(\Sigma\subset V){|0}}{\lambda_{-1}(T^*_0V)}\in Frac(K^\T(\pt))[y]$$
is $\Ne$-small.
\end{theorem}

\begin{remark}
For a non-isolated fixed point component $F$ the analogous statement is
\[
\lim_{\xi\to \infty}\left(\left.\frac{\tauy(\Sigma\subset V){|F}}{\lambda_{-1}(\nu^*_F)}\right|_{\alpha=\xi^s}\right)=\;\tauy(\Sigma\cap V^{s}_F\to F)\in K(F)[y]\,.
\]
\end{remark}

The proof is a direct application of Proposition \ref{prop:criterium} and \cite[Thm. 10]{WeBB}. For the convenience of the reader here we reformulate the proof of \cite{WeBB} with some simplifications and the notation adapted to our situation.

\begin{proof}By the additivity of $\tauy$ we can assume that $\Sigma$ is smooth.
According to Proposition \ref{prop:analytic} it is enough to prove that
\begin{equation}\label{eqn:lim1}
\lim_{\xi\to \infty} \left( \left.\frac{\tauy(\Sigma\subset V){|0}}{\lambda_{-1}(T^*_0V)}\right|_{\alpha=\xi^s}\right)<\infty.
\end{equation}
for almost all $s\in \Z^r$,
since $\Ne(\tauy(\Sigma\subset V)|0|_{\alpha=\xi^s }) = \pi_s( \Ne(\tauy(\Sigma\subset V)|0))$ for a generically chosen one parameter subgroup,
c.f. Proposition \ref{prop:crit2}.

 Let $s\in \Z^r$ be such that the $\alpha=\xi^s$ substitution does not map any of the weights of the representation to 0; the thus excluded $s$ vectors are contained in a finite union of hyperplanes.

The quantity $(\tauy(\Sigma\subset V){|0}/{\lambda_{-1}(T^*_0V)})|_{\alpha=\xi^s}$ can be interpreted as
$\tauy(\Sigma\subset V){|0}/{\lambda_{-1}(T^*_0V)}$ for the $\C^*$-action obtained from the $\T$-action of the theorem by composing it with $\kappa_s$.
Hence, in the remaining of the proof we work in $\C^*$-equivariant K-theory (without indicating this in our notation), in particular $K^{\C^*}(\{0\})=\Z[\xi^{\pm 1}]$ and want to show that
\[
\lim_{\xi\to \infty} \frac{\tauy(\Sigma\subset V){|0}}{\lambda_{-1}(T^*_0V)}<\infty.
\]
According to our assumption on $s$, the $\C^*$ action has a unique fixed point, $0$.

Let $Y$ be a proper normal crossing extension of the inclusion $\Sigma\to V$ (see Definition \ref{def:pnce}), i.e. $Y$ is
a smooth variety with $\C^*$ action, with an equivariant proper map $\eta: Y \to V$, such that $\eta^{-1}(\Sigma)\to\Sigma$ is an isomorphism
and $Y\setminus \eta^{-1}(\Sigma)$ is a normal crossing divisor.
For convenience let us assume that the map $Y\to V$ extends to a proper normal crossing extension $\hat{\eta}:\hat{Y}\to\hat V$ to a smooth compactification $\hat{V}$ of $V$.
The complement of $\hat{\eta}^{-1}(\Sigma)$ is a normal crossing divisor denoted by $D=\bigcup_{i=1}^k D_i$.
Then
\begin{equation}\label{eq:altern}\tauy(\Sigma\subset\hat{V})=\sum_{I\subset \{1,2,\dots,k\}}(-1)^{|I|}\eta_{I*}(\lambda_y(T^*D_I))\,,\end{equation}
where $D_I=\bigcap_{i\in I} D_i$, $D_\emptyset=\hat{Y}$ and $\eta_{I}$ is the restriction of $\hat\eta$ to $D_I$.
Note that by the locality property of motivic Chern classes $\tauy(\Sigma\subset{V})=\tauy(\Sigma\subset\hat{V})|V$, therefore $\tauy(\Sigma\subset{V})|0=\tauy(\Sigma\subset\hat{V})|0$.
First we show that
\[
\frac{\left(\eta_{I*}(\lambda_y(T^*D_I))\right){|0}}{\lambda_{-1}(T^*_0V)}
\]
is N-small. Since $V^{\C^*}=\{0\}$ the normal bundle to $V^{\C^*}$ is equal to $T_0V$.
By the localization\footnote{By \cite{Sum}, since $V$ is smooth, $0$ has an invariant affine neighbourhood.  We can assume that our arguments are applied to quasiprojective varieties.} theorem in K-theory see, \cite[Thm.~3.5]{Tho}, \cite[Thm. 5.11.7]{ChrissGinz}
\begin{equation}\label{eq:K-localization}
\frac{(\eta_{I*}(\lambda_y(T^*D_I))){|0}}{\lambda_{-1}(T^*_0V)}=
\sum
\eta_{F*}\left(\frac{\lambda_y(T^*D_I){|F}}{\lambda_{-1}(\nu^*_{F/D_I})}\right)\in Frac(K^\T(\pt))[y]\,.\end{equation}
The summation at the right hand side is over the components $F\subset D_I^{\C^*}$ such that $\eta(F)=\{0\}$.
Here $\nu_{F/D_I}$ is the normal bundle to the fixed point set component in $D_I$. To compute the  limit
we first evaluate it in $K^{\C^*}(F)[y]$. The variety $D_I$ is smooth and we obtain
\begin{equation}\label{eq:smoothlimit}
\lim_{\xi \to \infty}\frac{\lambda_y(T^*D_I){|F}}{\lambda_{-1}(\nu^*_{F/D_I})}=(-y)^{n_F}\lambda_y(T^*F),
\end{equation}
where $n_F$ is equal to the dimension of the subbundle of $\nu_{F/D_I}$ with negative weights.
This holds because
\[
\lim_{\xi \to \infty}\frac{1+y\xi^s}{1-\xi^s}=
\begin{cases}
-y & \text{for } s>0\\
1 &  \text{for } s<0,
\end{cases}
\]
see more details in \cite[Thm. 13]{WeBB}.
We conclude that the limit of the push-forward, which is the push-forward of the limit, is finite and equal to $(-y)^{n_F}\chi_y(F)$. This proves our first claim.

\medskip

Now we compute the limit explicitly. Consider the Bia{\l}ynicki-Birula cell $(D_I)_F:=(D_I)_F^{s}$ of $F$ in $D_I$ for our $\C^*$-action $\kappa_s$.
Observe that
\begin{equation}\label{eq:chiminus}
\chi_y((D_I)_F)=(-y)^{n_F}\chi_y(F)
\end{equation}
since by \cite{BB} the limit map $(D_I)_F\to F$ is a locally trivial  fibration in Zariski topology and the fiber is $\C^{n_F}$.
Hence the limit (\ref{eq:smoothlimit}) is equal to $\chi_y((D_I)_F)$.
We have
$$
\bigcup_{F\subset D_I^{\C^*} \,,\;\hat{\eta}(F)=\{0\}}(D_I)_F\;=\;\hat\eta_I^{-1}(V^{s})\,.
$$
Since $\hat\eta$ restricted to the inverse image of $\Sigma$ is an isomorphism, therefore
$$\chi_y(V^{s}_0\cap \Sigma)=\sum_{I\subset \underline{k}}(-1)^{|I|}
\chi_y(\eta_I^{-1}(V_0^{s}))=\sum_{I\subset \underline{k}}(-1)^{|I|}\sum_{F\subset D_I^{\C^*} \,,\;\hat{\eta}(F)=\{0\}}\chi_y((D_I)_F).$$
By (\ref{eq:altern}--\ref{eq:chiminus}) the former sum is the limit of $\tauy(\Sigma\subset V){|0}/\lambda_{-1}(T^*_0V)=\tauy(\Sigma\subset \hat V){|0}/\lambda_{-1}(T^*_0V)$ in $\C^*$-equivariant K-theory. This completes the proof.
\end{proof}

We will be concerned with a property of torus actions, which we will call {\em positive}, c.f. \cite{FP}.

\begin{definition}\label{def:positive}
We call a representation of $\T=(\C^*)^r$  {\em positive}, if any of the following equivalent conditions are satisfied.
\begin{itemize}
\item 
There exists an $s\in \Z^r$ such that the one parameter subgroup $\alpha(\xi)=(\xi^{s_1},\dots,\xi^{s_r})$ of $\T$ acts with positive weights;
\item the weights of the $\T$-action are contained in an open half-space $\{u\in \R^r : \pi_s(u)>0\}$ for some $s\in \Z^r$;
\item the convex hull of the weights of the $\T$-action does not contain $0\in \R^r$.
\end{itemize}
We call a $G$-representation {\em positive} if its restriction to a maximal torus is positive.
\end{definition}

 \begin{corollary}\label{co:strictinc}
If the action of $\T=(\C^*)^r$ on $T_0V$ is in addition positive and $0\not\in\Sigma$, then
\[
\Ne(\tauy(\Sigma\subset V){|0})\subset\Ne(\lambda_{-1}(T^*_0V))\setminus\{0\}.
\]
\end{corollary}

\begin{proof}
The containment in $\Ne(\lambda_{-1}(T^*_0V))$ is explicitly stated in Theorem \ref{thm:smallness} (even without the conditions on the action and on $\Sigma$), we need to prove that $0\not\in \Ne(\tauy(\Sigma\subset V){|0})$. We will prove this by showing that 0 is not in a projection of $\Ne(\tauy(\Sigma\subset V){|0})$.

We have $\lambda_{-1}(T^*_0V)=\prod(1-1/w_i)$ where $w_i$ are the (multiplicatively written) weights of the action.
Let $s\in \Z^r$ be such that the $\alpha=\xi^s$ substitution proves that the action is positive, and
\begin{equation}\label{eqn:bbb}
\Ne(\tauy(\Sigma\subset V)|0|_{\alpha=\xi^s }) = \pi_s( \Ne(\tauy(\Sigma\subset V)|0)).
\end{equation}
Such a choice is possible because the $s$ vectors proving positivity is open in the appropriate sense, and the $s$ vectors for which (\ref{eqn:bbb}) fails is a finite union of hyperplanes (c.f. Proposition \ref{prop:crit2}).

We have $\lambda_{-1}(T^*_0V)|_{\alpha=\xi^s}=\prod (1-1/\xi^{k_i})$, where $k_i$ are positive integers. Therefore
\begin{equation}\label{eqn:ize0}\Ne((\lambda_{-1}(T^*_0V))|_{\alpha=\xi^s})=\left[-\sum k_i,0\right] \subset \R_{\leq 0}.
\end{equation}
Note that, since the action through $\kappa_s$ has positive weights, the Bia{\l}ynicki-Birula cell of 0 consists only of 0: $$V^{s}_0=\{0\}\,.$$
By Theorem \ref{thm:smallness}
$$\lim_{\xi\to \infty}\left(\left.\frac{\tauy(\Sigma\subset V){|0}}{\lambda_{-1}(T^*_0V)}\right|_{\alpha=\xi^s}\right)=\;\chi_y(\Sigma\cap V^{s}_0)=0\,,$$
since $\Sigma\cap V^{s}_0=\emptyset$.
Using (\ref{eqn:ize0}) this implies that $0\not\in \Ne(\tauy(\Sigma\subset V)|0|_{\alpha=\xi^s})$, and hence by (\ref{eqn:bbb}) we have $0 \not \in \pi_s( \Ne(\tauy(\Sigma\subset V)|0))$, what we wanted to prove.
\end{proof}

The next two examples illustrate different aspects of Corollary \ref{co:strictinc}.

\begin{example}\rm \label{ex:good}
Consider the {\em positive} action of $(\C^*)^2$ on $\C^2$ given by $(\alpha,\beta)\cdot (x,y)=(\alpha\beta x,\alpha^3\beta^{-2}y)$; for example $s=(1,0)$ verifies that the action is positive. Let $X$ and $Y$ denote the first and the second coordinate axes respectively.
The additivity and normalization properties of $\tauy$ imply
\[
\tauy(\C^2)=(1+\frac{y}{\alpha\beta})(1+\frac{y}{\alpha^3\beta^{-2}}),
\quad
\tauy(X)=(1+\frac{y}{\alpha\beta})(1-\frac{1}{\alpha^3\beta^{-2}}),
\quad
\tauy(Y)=(1-\frac{1}{\alpha\beta})(1+\frac{y}{\alpha^3\beta^{-2}}),
\]
\[
\tauy(\{0\})=(1-\frac{1}{\alpha\beta})(1-\frac{1}{\alpha^3\beta^{-2}}),
\qquad
\tauy(X-\{0\})=\frac{1+y}{\alpha\beta}(1-\frac{1}{\alpha^3\beta^{-2}}),
\]
\[
\tauy(Y-\{0\})=(1-\frac{1}{\alpha\beta})\frac{1+y}{\alpha^3\beta^{-2}},
\qquad
\tauy(\C^2-\{0\})=(y^2-1)\frac{1}{\alpha^4\beta^{-1}}+(y+1)\frac{1}{\alpha\beta}+(y+1)\frac{1}{\alpha^3\beta^{-2}}.
\]
Here the restriction to 0 is an isomorphism and $\tauy$ is the same as $\tauy|0$.
It is instructive to verify that the Newton polytopes of $\tauy(X-\{0\})$, $\tauy(Y-\{0\})$, $\tauy(\C^2-\{0\})$, namely the convex hulls of $
\{(-4,1),(-1,-1)\},
\{(-4,1),(-3,2)\},
\{(-4,1),(-1,-1),(-3,2)\}
$
are contained in $\Ne(\lambda_{-1}({\C^2}^*))\setminus \{0\}=\Ne(\tauy(\{0\}))\setminus\{0\}=\ch(\{(-4,1),(-1,-1),(-3,2),(0,0)\}\setminus\{0\}$,
as illustrated in Figure 1.
\end{example}

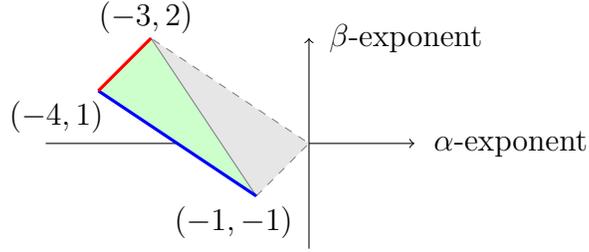
\begin{figure}[h!]
\centering
\begin{tikzpicture}[scale=.7]
\coordinate (A) at (0,0);
\coordinate (B) at (-3,2);
\coordinate (C) at (-4,1);
\coordinate (D) at (-1,-1);
\coordinate (E) at (2,0);
\coordinate (F) at (0,2);
\coordinate (G) at (-5,0);
\coordinate (H) at (0,-2);
\draw[->] [black, thin] (A) -- (E);
\draw[->] [black, thin] (A) -- (F);
\draw [black, thin] (A) -- (G);
\draw [black, thin] (A) -- (H);
\draw (5.5,0) node [left] {$\alpha$-exponent};
\draw (3.5,2) node [left] {$\beta$-exponent};
\draw (-3.7,.5) node [left] {$(-4,1)$};
\draw (-2,2.4) node [left] {$(-3,2)$};
\draw (-.1,-1.5) node [left] {$(-1,-1)$};
    \filldraw[dashed,
        draw=gray,
        fill=gray!20,
    ]          (D)
            -- (A)
            -- (B)
            -- (C);
\filldraw[draw=gray,
          fill=green!20,
    ]          (D)
            -- (B)
            -- (C);
\draw [blue, very thick] (C) -- (D);
\draw [red, very thick] (B) -- (C);
\end{tikzpicture}
\caption{Newton polytopes of Example \ref{ex:good}}
\end{figure}

\begin{example}\rm \label{ex:bad}
If the action of $\T$ is not positive, then the inclusion of Corollary \ref{co:strictinc} may not hold.
Consider the action of $\C^*$ on $\C^2$ given by $\alpha\cdot (x,y)=(\alpha x,\alpha^{-1}y)$.
We have
\[
\tauy(\C^2-\{0\})=
(1+\frac{y}{\alpha})(1+\frac{y}{\alpha^{-1}})-
(1-\frac{1}{\alpha})(1-\frac{1}{\alpha^{-1}})=
(y+1)\frac{1}{\alpha}+(y^2-1)+(y+1)\alpha.
\]
Hence, its Newton polytope is $[-1,1]$, which is not
contained in $\Ne(\lambda_{-1}(\C^{2*}))\setminus\{0\}=\Ne( (1-\frac{1}{\alpha})(1-\frac{1}{\alpha^{-1}}) )\setminus\{0\}=
[-1,1]\setminus\{0\}$. Note that here $\C^2 - \{0\}$ is an orbit of the standard $SL_2(\C)$ action.
\end{example}

\section{Axiomatic characterization of motivic Chern classes}

\subsection{Notations}

Let a linear group $G$ act on the smooth\quapr variety $V$.
Let $\Omega\subset V$ be an orbit, and let $x_\Omega\in \Omega$. We consider the stabilizer subgroup $G_{x_\Omega}$.
We will use the shorthand notation $G_\Omega=G_{x_\Omega}$. This subgroup is defined up to conjugation. We have
\[
K^G(\Omega)\simeq K^G(G/G_\Omega)\simeq K^{G_\Omega}(x_\Omega).
\]
Assume in addition, that the stabilizer $G_\Omega$ is connected, then
$$K^{G_\Omega}(x_\Omega)\simeq \Rep(L_\Omega)$$
 where $L_\Omega$ is the Levi group of $G_\Omega$, see \cite[\S5.2.18]{ChrissGinz}. We recall that by K-theory we mean the equivariant K-theory of locally free sheaves, which in this case (the base is smooth\quapr) coincides with the K-theory of coherent sheaves. Moreover since the representation ring of the reductive group is isomorphic to the representation ring of the maximal compact subgroup $L_{\Omega,c}$, we have
 \begin{equation}\label{comparison}K^G(\Omega)\simeq \Rep(L_{\Omega,c})\simeq K^{G_c}_{top}(\Omega)\,,\end{equation}
where $K^{G_c}_{top}$ is the topological K-theory defined for compact group actions by Segal, \cite{Seg}. Here $G_c\subset G$ is the maximal compact subgroup. It is shown in the Appendix that $K^G(V)\simeq K^{G_c}_{top}(V)$, where $V$ is a smooth \quapr variety on which $G$ acts with finitely many orbits without assumption about positivity and connected stabilizers.
\bigskip

Let  $\T_{\Omega}$ be the maximal torus in $G_{\Omega}$. We may assume that $\T_{\Omega}\subset \T$.
For a connected linear group $G$ the restriction homomorphism to the maximal torus $\T$ is an embedding and we identify $\Rep(G)$ with the invariants $\Rep(\T)^{W}$, where  $W$ is the Weyl group.
Let us introduce the maps
$$\phi_\Omega:K^G(V)\to \Rep(\T_\Omega)$$
by the composition of the restriction homomorphisms and natural isomorphisms
$$K^G(V)\to K^G(\Omega)\simeq K^{G_\Omega}(\pt)\to K^{\T_\Omega}(\pt)\simeq \Rep(\T_\Omega)\,.$$
The homomorphism $\phi_\Omega$ does not depend on the choices made.

The group $G_\Omega$ acts on the tangent space $T_{x_\Omega}\Omega$.
Consider a germ of a {\em normal slice} $S_\Omega$ of $\Omega$ at $x_\Omega$, that is a germ of a $\T_\Omega$-invariant smooth
subvariety of $V$ transverse to $\Omega$ at $x_\Omega$ of complementary dimension.
To construct a slice one can assume by \cite{Sum} that $V$ is a subvariety of a projective space $\P^N$ with a linear action of $\T_\Omega$. Further we can assume that $x_\Omega$ lies in an invariant affine chart of $\P^N$,  moreover that $x_\Omega=0\in \C^N$. There exists a linear subspace in $W\subset\C^N$, which is transverse to $V$ at $0$ and is $\T_\Omega$-invariant. That is so, since we can find a decomposition $\C^N=T_0\Omega\oplus W$ equivariantly with respect to the action of  $\T_\Omega$. The slice $S_\Omega$ at $x_\Omega=0$ is defined as a germ of $V\cap W$. The tangent space $T_{x_\Omega}S_\Omega$ will be denoted by $\nu(\Omega)\simeq T_0V/T_0\Omega$ and called the normal space to the orbit. As a representation of $\T_\Omega$ it does not depend on the choices made.

\subsection{The axiomatic characterization theorems}

We recall that for a representation $W$ of the torus $\T_\Omega$ we set
$$\lambda_y(W)=\sum_{k=0}^{\dim W} [\Lambda^kW] y^k\in \Rep(\T_\Omega)[y]$$ and hence $$\lambda_{-1}(W)=\sum_{k=0}^{\dim W} (-1)^k [\Lambda^kW]\in \Rep(\T_\Omega)\,.$$

\begin{lemma} \label{le:decomposition}
Let $\Theta, \Omega$ be orbits.
We have
\begin{equation}\label{eq:tauslice}
\phi_\Theta(\tauy(\Omega\subset V))=
\lambda_y(
T_{x_\Theta}(\Theta)^*)\cdot\tauy(\Omega\cap S_\Theta\subset S_\Theta){|\{x_\Theta\}}\in
\Rep(\T_\Theta)[y].
\end{equation}
\end{lemma}

\begin{proof}  Suppose $S_\Theta$ is a transverse slice to the orbit $\Theta$ at the point $x=x_\Omega\in\Theta$. Let the map $f:X\to V$ be a resolution of the pair $(\overline\Omega,\partial \Omega)$, where $\partial \Omega=\overline\Omega\setminus\Omega$, satisfying the condition of \S\ref{tau-definicja}, i.e.
$f^{-1}(\partial \Omega)=D=\bigcup_{i=1}^sD_i$ is a simple divisor with normal crossings.
Moreover assume that the resolution  is invariant with respect to automorphisms of $V$. In particular it is invariant with respect to the action of the group $G$.
Such a resolution is obtained by application of the Bierston-Milman \cite{BiMi} algorithm.
The slice $S_\Theta$ is automatically transverse to each orbit in a neighborhood of $x$.
For each $I\subset\underline{s}$ the image of the intersection $D_I=\bigcap_{i\in I}D_i$ is a union of orbits.
Precisely: if $y\in f(D_I)$, then $g\cdot y\in f(D_I)$ for every $g\in G$ since the action of $G$ lifts to $X$ preserving $D_I$. Moreover it follows that
the map $f$ restricted to $D_I$ is transverse to $S_\Theta$ (after possible shrinking of $S_\Theta$).
 Therefore $f^{-1}(S_\Theta)\cap D$ is a simple normal crossing divisor in $f^{-1}(S_\Theta)$ and $f$ restricted to ${f^{-1}(S_\Theta)}$ can be used to compute $\tauy(\Omega\cap S_\Theta\subset S_\Theta)$. After restriction to the point $x=x_\Theta$, for each $I\subset\underline{s}$ we have the projection formula
$$\phi_\Theta(f_*\lambda_y(T^*{D_I})){|x}=
\lambda_y(T^*_x\Theta)\cdot f_*\lambda_y(T^*{(D_I\cap f^{-1}(S_\Theta))}){|x}\in
\Rep(\T_\Theta)[y].
$$
Using the formula (\ref{alternating-definition}) of \S\ref{tau-definicja} we obtain the claim.\end{proof}

\begin{remark} The restriction formula holds in a more general situation, not necessarily
when we deal with a space with a group action, stratified by orbits.
It is a version of Verdier-Riemann-Roch. For
CSM-classes an analogous formula was discussed by Yokura who gave a proof for
smooth morphism \cite[Thm. 2.2]{YoVRR}. His argument applied to the natural
map $G\times^{\T_\Theta}S_\Theta\to V$ is valid for equivariant motivic
Chern classes as well. The proof of Verdier-Riemann-Roch by Sch\"urmann \cite[Cor
0.1]{Sch} also works in equivariant K-theory,  provided that we apply the
specialization functor in the category of
equivariant mixed Hodge modules. Another proof for CSM-classes was given
by Ohmoto \cite[Prop. 3.8]{O2}\footnote{In an earlier version of Ohmoto's paper (arXiv:1309.0661v1) the proof was attained is a similar way as ours.}.
Nevertheless  the
exact statement and a proof for equivariant K-theory is not  accessible in the literature,
therefore we gave it here taking advantage of the group action.
 \end{remark}

\begin{theorem} \label{thm:intchar}
Let $V$ be a smooth variety (not necessarily complete) on which a linear algebraic group $G$ acts with finitely many orbits.
Let us assume that
\begin{enumerate}[(*)]
\item
for all $\Omega$ the action of $\T_\Omega$ on $\nu(\Omega)$ is positive (see Definition \ref{def:positive}).
\end{enumerate}
Then the following properties hold
\begin{enumerate}[(i)]
\item \label{iprinc}
$\phi_\Omega(\tauy(\Omega\subset V))=\lambda_y(T^*_{x_\Omega}\Omega)\cdot\lambda_{-1}(\nu(\Omega)^*) \in \Rep(\T_\Omega)[y]$,
\item \label{idivis}
$\phi_\Theta(\tauy(\Omega\subset V))$ is divisible by $\lambda_y(T(\Theta)^*)$ in $\Rep(\T_\Theta)[y]$,
\item \label{idegree}
if $\Theta\not= \Omega$ then $\Ne(\phi_\Theta(\tauy(\Omega\subset V))/\lambda_y(T(\Theta)^*)\subset \Ne(\lambda_{-1}(\nu(\Theta)^*))\setminus\{0\}$.
\end{enumerate}
\end{theorem}

\begin{remark} \rm
Observe that using (\ref{eqn:CL}), (\ref{eqn:Nprop}), and the fact that $0\in \Ne(\lambda_y(T(\Theta)^*))$ condition (\ref{idegree}) is equivalent to
\begin{enumerate}[(i)]
\setcounter{enumi}{3}
\item {\em  if $\Theta\not= \Omega$ then $\Ne(\phi_\Theta(\tauy(\Omega\subset V)))\subset    \Ne(\lambda_y(T(\Theta)^*)\lambda_{-1}(\nu(\Theta)^*))\setminus\{0\}$,}
\end{enumerate}
or, what is the same, to
\begin{enumerate}[(i)]
\setcounter{enumi}{4}
\item {\em  if $\Theta\not= \Omega$ then $\Ne(\phi_\Theta(\tauy(\Omega\subset V)))\subset
\Ne(\phi_\Theta(\tauy(\Theta\subset V)))\setminus\{0\}$.}
\end{enumerate}
\end{remark}

\begin{proof}
{\noindent  (i)~~} Since $\Omega$ is smooth
\[
\tauy(\id_\Omega))=\lambda_y(T_{x_\omega}^*\Omega)\in K^G(\Omega)[y].
\]
The composition of inclusion into $V$ and restriction introduces the factor $\lambda_{-1}(\nu(\Omega)^*)$, the K-theoretic Euler class of $\nu(\Omega)$.

{\noindent  (ii)~~} Follows from Lemma \ref{le:decomposition}.

{\noindent  (iii)~~} By Corollary \ref{co:strictinc} we have
\begin{equation}\label{eqn:temp5}
\Ne(\tauy(\Omega\cap S_\Theta\subset S_\Theta){|x_\Theta})\subset\Ne(\lambda_{-1}(\nu(\Theta)^*))\setminus\{0\}.
\end{equation}
Applying the multiplicative property (\ref{eqn:Nprop}) of Newton polytopes and formula (\ref{eq:tauslice}) we obtain the claim.
\end{proof}

\begin{theorem} \label{thm:axiomatic}
Suppose that $V$ is a smooth algebraic $G$-variety which is the union of finitely many orbits. Assume that the stabilizers of the orbits are connected and the action satisfies the positivity condition (*).
Then the properties (\ref{iprinc})-(\ref{idegree}) of Theorem \ref{thm:intchar} determine $\tauy(\Omega\subset V)\in K^G(V)[y]$.
\end{theorem}

Our proof is an adaptation of the arguments in \cite[Prop. 9.2.2]{O}, \cite[Section 3.2]{RTV2}.

\begin{proof}Let us fix a linear order of orbits such that
$$V_{\succ\Theta}=\bigsqcup_{\Omega\succ\Theta}\Theta$$ is an open set in $V$. Let $V_{\succeq\Theta}=\bigsqcup_{\Omega\succeq\Theta}\Theta$.
Consider
the long exact sequences for the pairs $(V_{\succeq\Theta},V_{\succ\Theta})$.
Since the stabilizers of the orbits are connected we have
$$K^G(\Theta)=K^{G_\Theta}(\pt)=\Rep(\T_\Theta)^{W_\Theta}\,.$$
Let $\iota_\Theta:\Theta\to V_{\succeq\Theta}$ be the inclusion.
The map $(\iota_\Theta)_*$ composed with $\iota_\Theta^*$ is the multiplication by $\lambda_{-1}(\nu^*_\Theta)$, which is nonzero element by the positivity assumption. Since
$\Rep(\T_\Theta)^{W_\Theta}$ is a domain, the map $(\iota_\Theta)_*$  is injective:
\begin{equation}\label{eq:sequence}
\xymatrix{0\ar[r]&K^G(\Theta)\ar[r]^{(\iota_\Theta)_*} &K^G(V_{\succeq\Theta})\ar[r]&K^G(V_{\succ\Theta})\ar[r]&0}\\
\end{equation}
Another argument showing exactness in a much more general situation is given in the Appendix, the proof of Theorem \ref{K-LES}.
\medskip

Suppose that $\tau(\Omega)$ satisfies the conditions {(i)-(iii)}. Then the difference $\delta=\tau(\Omega)-\tauy(\Omega)$ satisfies
\begin{enumerate}[(i')]
\item \label{iprinc'}
$\phi_\Omega(\delta)=0 \in \Rep(\T_\Omega)[y]$,
\item \label{idivis'}
$\phi_\Theta(\delta)$ is divisible by $\lambda_y(T^*\Theta)$ in $\Rep(\T_\Theta)[y]$
\item \label{idegree'}
$\Ne(\delta)\subsetneq \Ne( \lambda_y(T^*\Theta)\cdot\lambda_{-1}(\nu^*_\Theta))$ for all orbits, including $\Theta=\Omega$.
\end{enumerate}
The condition (iii') holds because under the positivity assumption (*) the point 0 is a vertex of $\lambda_{-1}(\nu^*_\Theta)$, hence it does not belong to
$$\Ne(\phi_\Theta(\delta)/\lambda_y(T^*\Theta))\subset conv.hull(\Ne(\phi_\Theta(\tau))/\lambda_y(T(\Theta)^*))\cup\Ne(\phi_\Theta(\tauy(\Omega)))/\lambda_y(T(\Theta)^*)))\,.$$
We argue by the induction on  orbits that
\begin{equation}\label{eq:zero}\delta|V_{\succeq \Theta}=0.\end{equation}
If $\Theta$ is the open orbit, then $V_{\succeq\Theta}=\Theta$ and $\lambda_{-1}\nu(\Theta)=1$. The property (iii') implies that $\Ne(\phi_\Theta(\delta))=\emptyset$, thus $\phi_\Theta(\delta)=0\in \Rep(\T_\Theta)[y]$. Since $G_\Theta$ is connected the restriction map $\Rep(G_\Theta)\to \Rep(\T_\Theta)$ is injective. Hence $\delta{|\Theta}=0$.

Suppose by the inductive assumption that $\delta|V_{\succ \Theta}=0$. We prove that $\delta|V_{\succeq \Theta}=0$. By the exact sequence (\ref{eq:sequence}) the class $\delta|V_{\succeq \Theta}$ is the image of an element $$\gamma\in K^G(\Theta)[y]\hookrightarrow K^{\T_\Theta}(\pt)[y]\,.$$
We have  $$\iota_{\Theta}^*(\iota_{\Theta})_*(\gamma)=\lambda_{-1}(\nu^*_\Theta)\cdot\gamma\,.$$
On the other hand $\iota_{\Theta}^*(\iota_{\Theta})_*(\gamma)$ is equal to the restriction of $\delta$ to $\Theta$.
By (ii') it is divisible by $\lambda_y(T^*\Theta)$ in $\Rep(\T_\Theta)[y]$.
The Laurent polynomials  $\lambda_y(T^*\Theta)$ and $\lambda_{-1}(\nu^*_\Theta)$  are coprime in $\Rep(\T_\Theta)[y]$. Thus $\phi_\Theta(\delta)$ is divisible by $\lambda_y(T^*\Theta)\cdot\lambda_{-1}(\nu^*_\Theta)$. This contradicts the proper inclusion of Newton polytopes. The only possibility is that $\phi_\Theta(\delta)=0$, and as in the initial step we conclude that $\delta|\Theta=0$. From the exactness of the sequence (\ref{eq:sequence}) it follows that $\delta|V_{\succeq \Theta}=0$.
\end{proof}

\begin{remark} \rm \label{rem:comparison}
The axiomatic characterization of motivic Chern classes in this section is motivated by the axiomatic characterization of certain K theoretic characteristic classes called K theoretic stable envelopes in \cite[Section 9.1]{O}. The similarity of the axiom systems imply certain coincidences that we explain now. We do not attempt to give a detailed setup of the theory of stable envelopes, the explanation we are giving now requires some familiarity with works of Okounkov, see for example \cite{MO1, MO2, OS, AO1, AO2} or \cite{GRTV, RTV1, RTV2, RTV3, RV1, RV2, FRV, SZZ}. The relation between motivic Chern classes and stable envelopes is also studied in \cite{AMSS2}.

Stable envelopes are defined in the following context: on a symplectic complex algebraic variety (usually a Nakajima quiver variety) $Y$   a group $\T \times \C^*$ acts naturally, where $\T$ leaves the symplectic form invariant and $\C^*$ scales it by a character. Stable  envelopes are associated with components of the $\T$ fixed points, and they also depend on an extra parameter, called slope. For a fixed slope parameter the stable envelopes live in $K^{\T\times \C^*}(Y)$. They are defined by three axioms very similar to our axioms for motivic Chern classes: normalization, support, and Newton polytope axioms.

In this paper we are studying $G$ equivariant motivic Chern classes for possibly singular or locally closed invariant subvarieties of an arbitrary smooth $G$-manifold $X$. These classes live in $K^G(X)[y]$.

In certain situations both notions make sense, and the rings they live in can be identified such that the axioms match. Hence in these cases the two notions (after the identifications) are the same. Here are the assumptions to make to guarantee that the two notions match. First, the variety $Y$ needs to be a cotangent space of our smooth variety $X$, such that $G$ leaves $X$ invariant and $\C^*$ acts by multiplication in the fiber. Suppose also the connected and solvable group $G$, with maximal torus $\T$, acts on $X$ with finitely many orbits with one $\T$ fixed point in each orbit, and the induced $\T$ action is positive. Further assume that all stabilizer subgroups are connected. Then stable envelopes and motivic Chern classes both live in
\[
K^{\T\times \C^*}(Y)=K^{\T\times \C^*}(TX)=K^{\T\times \C^*}(X)=K^{\T}(X)[y^{\pm 1}]
=
K^G(X)[y^{\pm 1}],
\]
(see \cite[Sect. 5.2.18]{ChrissGinz} for the last identification). Not only they live in the same ring and both are assigned to the finitely many $\T$ fixed points, but their defining axioms also (almost) match. Namely, the normalization and support axioms of both notions are the same. The Newton polytope axiom of \cite[Section 9.1]{O} requires that the small convex polytope remains inside the larger one even if shifted slightly towards the origin (at least for a specific choice of slope parameter \cite[9.1.9]{O}, called antidominant alcove). This implies that the small polytope is in the large polytope minus the origin, which is our Newton polytope axiom. Since one axiom implies the other one, the two notions they characterize---in the special case we consider, and after the identifications we made---coincide.

This argument implies that---in the situation when both notions are defined---a stronger Newton polytope axiom also holds for the motivic Chern classes. This phenomenon, however, only holds in the special situation when both stable envelopes and motivic Chern classes are defined and identified. For motivic Chern classes of general varieties only the weak Newton polytope containment holds, see e.g. Example \ref{ex:good}.
\end{remark}

\subsection{On the conditions of Theorem \ref{thm:axiomatic}}

Theorem \ref{thm:axiomatic} provides an axiomatic characterization of $\tauy(\Omega\subset V)$ if certain conditions hold. One of the conditions is that the stabilizer subgroups of the orbits are connected. The following examples show that this condition is indeed required.

\begin{example}\rm
Let the torus $\T$ act on $V=\T/H$ where $H$ is a finite subgroup of $\T$. The action has a unique orbit $\Omega$.
We have $K^\T(V)[y]=Rep(H)[y]$. Since the maximal torus $\T_\Omega$ of the stabilizer $G_\Omega$ is trivial, we have $Rep(\T_\Omega)[y]=\Z[y]$, and hence the map $\phi_\Omega$ is the obvious forgetful map $Rep(H)[y]\to \Z[y]$. Since this map is not injective (not even after tensoring with $\Q$!) the uniqueness statement in Theorem \ref{thm:axiomatic} does not hold. On the other hand the values of $\phi_\Omega$ and of the Chern character are determined by (i)-(iii), since  here the target groups do not contain much information.
\end{example}

\begin{example}\rm Let $X=\C$, $G=\T=\C^*$ acting on $\C$ by $t\cdot z=t^nz$.
Let $\Omega=\C\setminus\{0\}$, $\Theta=\{0\}$. We have a short exact sequence of $\Rep(\C^*)$-modules
$$\xymatrixcolsep{3.5pc}\xymatrix{0\ar[r]&K^{\C^*}(\Theta)\ar@{=}[d]
\ar[r]^{(\iota_\Theta)_*
}&K^{\C^*}(\C)\ar@/^2pc/[l]_{\iota_\Theta^*}\ar[r]^{\iota_\Omega^*}&K^{\C^*}(\Omega)\ar@{=}[d]\ar[r]& 0\,.\\
&\Rep(\C^*)&&\Rep(\Z/(n))}$$
We have $\T_\Omega=1$, $\T_\Theta=\T$ and the restriction $\phi_\Theta=\iota_\Theta^*:K^\T(\C)\to K^\T(\Theta)$ is an isomorphism. Naming the generator of $\Rep(\C^*)$ by $\alpha$ we obtain that the composition $\iota_\Theta^*(\iota_\Theta)_*$ is the multiplication  by $1-\alpha^{-n}$, and the sequence
$$\xymatrixcolsep{3.5pc}\xymatrix{0\ar[r]& \Z[\alpha,\alpha^{-1}]~~~\ar[r]^{(1-\alpha^{-n})\cdot}&~~~\Z[\alpha,\alpha^{-1}]\ar[r] &\Z[\alpha]/(\alpha^n-1)\ar[r]& 0\,.}$$
Suppose $\tau\in K^\T(\C)$ satisfies (i)--(iii).
Property (i) fixes the value $\phi_\Omega(\tau)=1+y\in K^{\{1\}}(x_\Omega)[y]=\Z[y]$. Property (ii) says that $\phi_\Theta(\tau)$ is divisible by $1$. Property (iii) gives the inclusion of Newton polytopes $\Ne(\phi_\Theta(\tau))\subset [-n,1]$.

Therefore the classes of the form $(1+y)\alpha^{-k}$ with any $k\in [1,n]$ satisfy (i)-(iii).
Hence for $n>1$ the class $\tau$ is not determined by (i)-(iii) in $K^\T(\C)$, nor the values of $\phi_\Theta(\tau)$, nor the value of the Chern character.
\end{example}

We believe that a modification of Theorem \ref{thm:axiomatic}
should be true with {\em non-necessarily connected stabilizers}, but some condition on the discrete part of $\tauy(\Omega\subset V)|\Theta$ should be imposed to guarantee uniqueness.
Moreover, we believe that there is a version of Theorem \ref{thm:axiomatic} dealing with {\em non-necessarily positive} actions.
We plan to study these extensions
in the future.

\section{Motivic Chern classes of Schubert cells in partial flag varieties}
\label{sec:compact}
In this section we reinterpret results of \cite{RTV3} to our settings, and thus we obtain explicit rational function representatives of motivic Chern classes of Schubert cells in partial flag varieties.

\subsection{The partial flag variety and its Schubert cells}
Let  $N,n$ be non-negative integers, and let $\mu=(\mu_1,\ldots,\mu_N)\in \N^N$, such that $\sum_{i=1}^N \mu_i=n$. Define $\mu^{(j)}=\sum_{i=1}^j \mu_i$. Consider the flag variety $\Fl_\mu$ parameterizing flags of subspaces
\[
0=V_0 \subset V_1\subset V_2\subset  \ldots \subset V_{N-1} \subset V_N=  \C^n
\]
with $\dim V_j=\mu^{(j)}$. The tautological bundle of rank $\mu^{(j)}$, whose fiber over $V_\bullet$ is $V_j$, will be denoted by $\F_{j}$.

Let $I = (I_1,...,I_N)$ be a partition of $\{1,...,n\}$ into disjoint subsets $I_1,...,I_N$ with $|I_j|=\mu_j$. The set of such $I$'s will be denoted by $\I_\mu$. We will use the following notation: For $I\in \I_\mu$ let $I^{(j)}=\cup_{i=1}^j I_i$ and $I^{(j)}=\{i^{(j)}_1<\ldots<i^{(j)}_{\mu^{(j)}}\}$.

For $I\in \I_\mu$ define the Schubert cell
\[
\Omega_I =
\{V_{\bullet}\in \Fl_\mu:  \dim(V_{p} \cap \C_\text{last}^q ) = \#\{i \in I_1 \cup \ldots, \cup I_p: i> n-q\}, \forall p,q\},
\]
where $\C_\text{last}^q$ is the span of the last $q$ standard basis vectors in $\C^n$. We have $\codim \Omega_I=\#\{(a,b)\in \underline{n}\times \underline{n}: a>b, a\in I_j, b\in I_k,j<k\}$.

\subsection{The equivariant K-ring of $\Fl_\mu$} \label{sec:Kring}

The standard action of the torus $\T=(\C^*)^n$ on $\C^n$ induces an action on $\Fl_\mu$ and the bundles $\F_j$. Denote the K-theoretic Chern roots of $\F_j$ by $\alpha^{(j)}_a$ ($a=1,\ldots,\mu^{(j)}$), that is $\sum_a \alpha^{(j)}_a=\F_j$ in $K^\T(\Fl_\mu)$. Observe that $\alpha^{(N)}_a$ are the Chern roots of the trivial $\C^n$ bundle with the standard $\T$-action, that is, the $\alpha$-variables with upper index $N$ are variables in $K^\T(\pt)$. The algebra $K^\T(\Fl_\mu)$ is a certain quotient of the Laurent polynomial ring
\begin{equation}\label{eqn:thering}
\Z\left[(\alpha^{(j)}_a)^{\pm1}\right]^{S_{\mu^{(1)}} \times \ldots\times S_{\mu^{(N-1)}}}_{j=1,\ldots,N, a=1,\ldots,\mu^{(j)}},
\end{equation}
by an ideal---not needed in this paper---whose generators express the fact that the bundles $\F_j/\F_{j-1}$ have rank $\mu_j$.

Another way of describing the algebra $K^{\T}(\Fl_\mu)$ is equivariant localization. The torus fixed points of $\Fl_\mu$ are flags of coordinate subspaces, they are also parameterized by $\I_\mu$. The restriction homomorphism $r_I:K^\T(\Fl_\mu) \to K^\T(x_I)$ to the fix point $x_I$ corresponding to $I\in \I_\mu$ is the substitution
\begin{equation}\label{eqn:egymeg}
r_I:\alpha^{(j)}_a \mapsto \alpha^{(N)}_{i^{(j)}_a}\qquad\text{for } j=1,\ldots,N-1, a=1,\ldots,\mu^{(j)}.
\end{equation}

\begin{example} \rm
We have
\[
K^{\T}(\Fl_{(1,1,1)})=\Z\left[\left(\alpha^{(1)}_1\right)^{\pm1};\left(\alpha^{(2)}_1\right)^{\pm1},
\left(\alpha^{(2)}_2\right)^{\pm1};\left(\alpha^{(3)}_1\right)^{\pm1},
\left(\alpha^{(3)}_2\right)^{\pm1},\left(\alpha^{(3)}_3\right)^{\pm1}\right]^{S_2}/\text{ideal},
\]
and the restriction map
\[
K^{\T}(\Fl_{(1,1,1)}) \to K^{\T}(x_{(\{u\},\{v\},\{w\})})=
\Z\left[\left(\alpha^{(3)}_1\right)^{\pm1},\left(\alpha^{(3)}_2\right)^{\pm1},\left(\alpha^{(3)}_3\right)^{\pm1}\right]
\]
to $x_{(\{u\},\{v\},\{w\})}$, where $(u,v,w)$ is a permutation of $(1,2,3)$, is induced by
\[
\alpha^{(1)}_1\mapsto \alpha^{(3)}_u, \qquad
\alpha^{(2)}_1\mapsto \alpha^{(3)}_{\min(u,v)},
\alpha^{(2)}_2\mapsto \alpha^{(3)}_{\max(u,v)}.
\]
Because of the $S_2$-symmetry in $\alpha^{(2)}_1,\alpha^{(2)}_2$, the same map is obtained by
\[
\alpha^{(1)}_1\mapsto \alpha^{(3)}_u, \qquad
\alpha^{(2)}_1\mapsto \alpha^{(3)}_{u},
\alpha^{(2)}_2\mapsto \alpha^{(3)}_{v}.
\]

\end{example}

\subsection{Weight functions, modified weight functions} \label{sec:weightfunctions}
 For $I\in \I_\mu$, $j=1,\ldots,N-1$, $a=1,\ldots,\mu^{(j)}$, $b=1,\ldots,\mu^{(j+1)}$ define
\[
\psi_{I,j,a,b}(\xi)=\begin{cases}
1-\xi & \text{if } i^{(j+1)}_b<i^{(j)}_a \\
 (1+y)\xi & \text{if }  i^{(j+1)}_b=i^{(j)}_a \\
1+y\xi & \text{if }  i^{(j+1)}_b>i^{(j)}_a.
\end{cases}
\]
\begin{remark} \label{rem:coinc}
It is worth comparing the values of this function with the fundamental calculation in Section \ref{sec:fundcalc}.
\end{remark}

Define the ``weight function''
\[
W_I=
\Sym_{S_{\mu^{(1)}} \times \ldots \times S_{\mu^{(N-1)}}} U_I
\]
where
\[
U_I=
\prod_{j=1}^{N-1} \prod_{a=1}^{\mu^{(j)}} \prod_{b=1}^{\mu^{(j+1)}}
\psi_{I,j,a,b}(\alpha^{(j)}_a/\alpha^{(j+1)}_b )
\cdot
\prod_{j=1}^{N-1} \prod_{1\leq a <  b\leq \mu^{(j)}} \frac{ 1+y\alpha^{(j)}_b/\alpha^{(j)}_a}{1-\alpha^{(j)}_b/\alpha^{(j)}_a}.
\]
Here the symmetrizing operator is defined by
\[
\Sym_{S_{\mu^{(1)}} \times \ldots \times S_{\mu^{(N-1)}}}=
\sum_{\sigma\in S_{\mu^{(1)}} \times \ldots \times S_{\mu^{(N-1)}}} U_I(\sigma(\alpha^{(j)}_a))
\]
where the $j$th component of $\sigma$ (an element of $S_{\mu^{(j)}}$) permutes the $\alpha^{(j)}$ variables.
For
\[
e_\mu=
\prod_{j=1}^{N-1}
\prod_{a=1}^{\mu^{(j)}}
\prod_{b=1}^{\mu^{(j)}}
(1+y\alpha^{(j)}_b/\alpha^{(j)}_a)
\]
define the ``modified weight function''
\[
\tilde{W}_I=W_I/e_\mu.
\]
Observe that $\tilde{W}_I$ is not a Laurent polynomial, but rather a ratio of two such.

\begin{lemma}\label{lem:okay}
The $r_J$-image  (c.f. (\ref{eqn:egymeg})) of $\tilde{W}_I$ for any $J\in \I_\mu$ is a Laurent polynomial. There exists a {\em Laurent polynomial} in the ring (\ref{eqn:thering}) whose $r_J$-images are the same as those of $\tilde{W}_I$ for all $J$. The class in $K^\T(\Fl_\mu)[y]$ of this other Laurent polynomial will be denoted by $[\tilde{W}_I]$.
\end{lemma}

\begin{proof}
The statement is a special case of \cite[Lemma~3.3 and Section 5.2]{RTV3}, see also Remark~\ref{rem:expl}.
\end{proof}

\begin{theorem}\label{thm:cptmCW}
\label{thm:cpt}
We have
\[
\tauy(\Omega_I\subset \Fl_\mu)=[\tilde{W}_I] \in K^\T(\Fl_\mu).
\]
\end{theorem}

\begin{proof}
The $r_J$-images of $[\tilde{W}_I]$ satisfy the axioms for $\tauy(\Omega_I\subset \Fl_\mu)$ in Theorems \ref{thm:intchar}, \ref{thm:axiomatic}, for the $B_n^-$-action on $\Fl_\mu$. This statement is a special case of Lemma~3.5, Lemma~3.6, and Theorem~3.9 of \cite{RTV3}, see also Remark~\ref{rem:expl}
\end{proof}

\begin{remark} \label{rem:expl}
In the proof of Lemma \ref{lem:okay} and Theorem \ref{thm:cptmCW} we cited ``special cases'' of results in \cite{RTV3}. Here let us explain to the reader what needs to be ``specialized'' in the results of  \cite{RTV3} for the purpose of our proofs. In \cite[Section 3.1]{RTV3} ``trigonometric weight functions'' $W_{\sigma,I}^\Delta$ and ``modified trigonometric weight functions'' $\tilde{W}_{\sigma,I}^{\Delta}$ are defined, depending on three combinatorial parameters: $I\in \I_\mu$, $\sigma\in S_n$, and an ``alcove'' $\Delta$. Our weight functions $W_I$ and modified weight functions $\tilde{W}_I$ only depend on the parameter $I\in \I_\mu$. The fact is that our weight functions are special cases of $W_{\sigma,I}^{\Delta}$ and $\tilde{W}_{\sigma,I}^{\Delta}$ for special choices of $\sigma$ and $\Delta$. The choice of $\sigma$ is $\sigma=\id$. The alcove $\Delta$ is characterized by a sequence of integers $m_{i,j}$ in \cite{RTV3}. The specialization we need for our paper is $m_{i,j}=-1$ for all $i,j$. For these choices, the cited Lemmas and Theorem of \cite{RTV3} prove axioms (i), (ii), and a property stronger than (iii) of Theorem \ref{thm:intchar}.
\end{remark}

\begin{remark}
In light of Remark \ref{rem:expl} it is natural to ask how to modify the notion ``motivic Chern class of the Schubert cell $\Omega_I$'' so that it equals the formula $\tilde{W}^\Delta_{\sigma,I}$ of \cite{RTV3}, not just its special case $\tilde{W}^{(m_{i,j}=-1)}_{\id,I}$. The role of $\sigma\in S_n$ is simple, it corresponds to choosing a different reference full flag when defining the Schubert cells.
We plan to explore the role of the alcove $\Delta$ (also called dynamical parameters, or {\em slope} in works of Okounkov), in the future. In the previous works \cite{RTV2} the alcove $(m_{i,j}=0)$ was studied. The weight function $\tilde{W}^{(m_{i,j}=0)}_{\id,I}$ differs from $\tilde{W}^{(m_{i,j}=-1)}_{\id,I}$ by the equivariant Grothendieck duality and a normalizing factor depending on the dimension of the cell.
\end{remark}

\begin{example} \label{ex:Gr24} \rm
Consider $\mu=(2,2)$ and hence $\Fl_\mu=\Gr_2\C^4$. For the motivic Chern class of the Schubert cell corresponding to $(\{1,3\},\{2,4\})\in \I_{(2,2)}$ we get
\[
\tauy(\Omega_{\{1,3\},\{2,4\}}\subset \Gr_2\C^4)=
\left[
\frac{1}{ (1+y)^2(1+\frac{y\alpha_2}{\alpha_1})(1+\frac{y\alpha_1}{\alpha_2})}\left( U_1+U_2\right)
\right]
\]
where
\[
U_1=
\frac{(1+y)^2\frac{\alpha_1\alpha_2}{\beta_1\beta_3}(1+\frac{y\alpha_1}{\beta_2})(1+\frac{y\alpha_1}{\beta_3})(1+\frac{y\alpha_1}{\beta_4})(1-\frac{\alpha_2}{\beta_1})(1-\frac{\alpha_2}{\beta_2})(1+\frac{y\alpha_2}{\beta_4})(1+\frac{y\alpha_2}{\alpha_1})}
{(1-\frac{\alpha_2}{\alpha_1})}
\]
and $U_2(\alpha_1,\alpha_2)=U_1(\alpha_2,\alpha_1)$. Here we used the short-hand notation $\alpha_i=\alpha_{1,i}$, $\beta_i=\alpha_{2,i}$. As seen, we represented $\tauy(\Omega_{\{1,3\},\{2,4\}})$ with a rational function. Yet, its restrictions to all Schubert cells (or, equivalently, torus fixed points) are Laurent polynomials. For example the restriction to $\Omega_{\{1,2\},\{3,4\}}$ is 0, the restriction to
$\Omega_{\{1,3\},\{2,4\}}$ is
$
\left(1+{y\beta_1}/{\beta_2}\right)\left(1+{y\beta_1}/{\beta_4}\right)\left(1-{\beta_3}/{\beta_2}\right)\left(1+{y\beta_3}/{\beta_4}\right),
$
and the restriction to $\Omega_{\{3,4\},\{1,2\}}$ is
\[
\frac{(1+y)\beta_4}{\beta_1^2\beta_2^2}
\left(
y^2(\beta_1\beta_2\beta_3-\beta_3^2\beta_4) \hskip 10 true cm   \right.
\]
\[\
\left.
\hskip 3 true cm
+y(2\beta_1\beta_2\beta_3+\beta_1\beta_2\beta_4-\beta_1\beta_3\beta_4+\beta_2^2\beta_3-\beta_2\beta_3^2-2\beta_2\beta_3\beta_4)+
\beta_1\beta_2^2-\beta_2\beta_3\beta_4
\right).
\]
According to the Newton polytope axiom Theorem \ref{thm:intchar}(\ref{idegree}), the Newton polytope of this last expression needs to be contained in the Newton polytope of $(1-\beta_3/\beta_1)(1-\beta_3/\beta_2)(1-\beta_4/\beta_1)(1-\beta_4/\beta_2)$ (minus the origin); this containment is illustrated in Figure \ref{fig:convex}.
\begin{figure}
\begin{center}
\obrazek{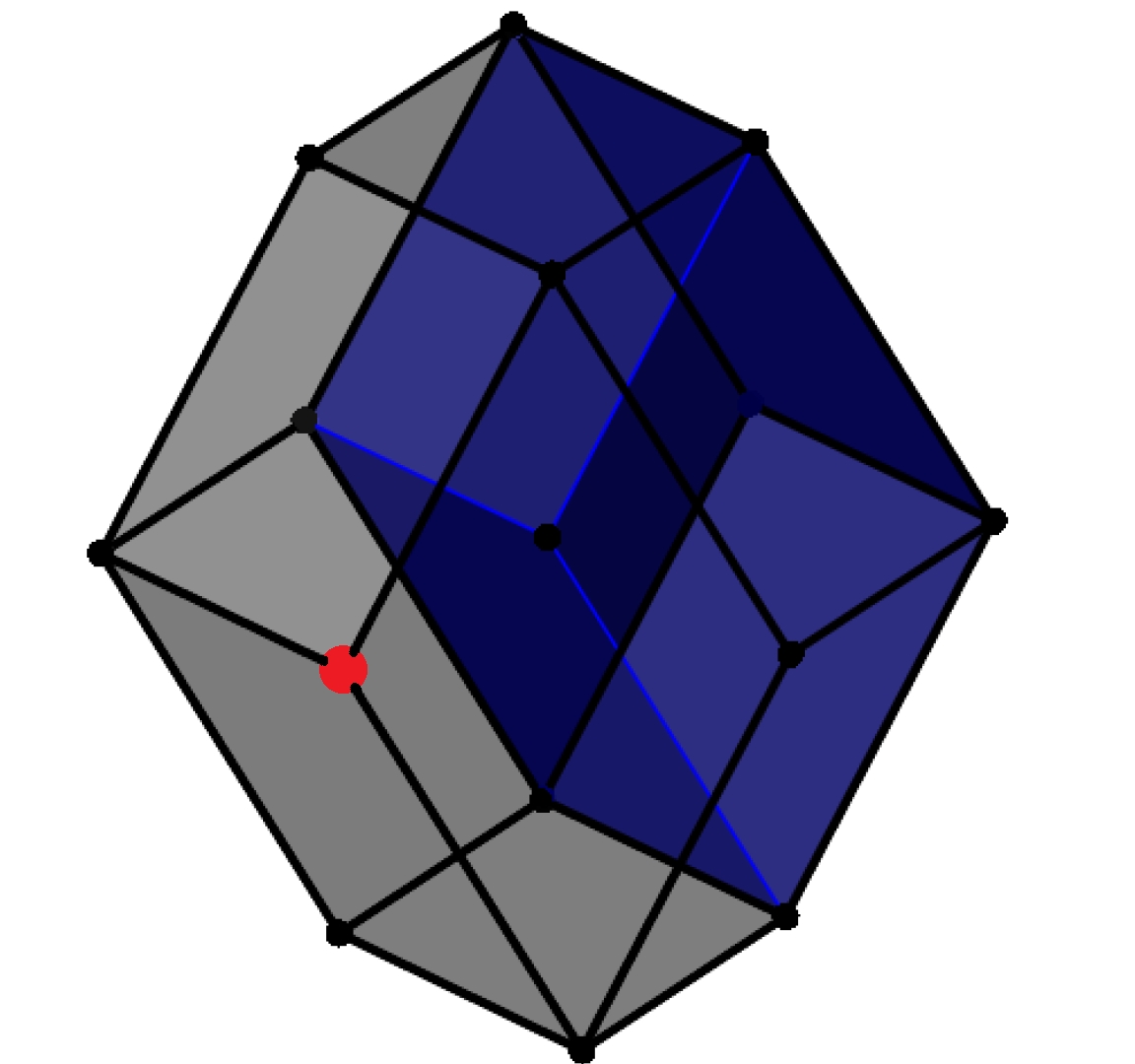}{3.2in}
\caption[]{Newton polytopes of Example \ref{ex:Gr24}}
\label{fig:convex}
\end{center}
\end{figure}
\end{example}

It is natural to think that Theorem \ref{thm:cptmCW} has a geometric proof, which, for example, explains Remark \ref{rem:coinc}. In fact, this is indeed the case, one may use the traditional resolution method to achieve Theorem \ref{thm:cptmCW}. Although such a proof---based on the fact that flag varieties can be obtained as GIT (or symplectic) quotients of vector spaces---has advantages, we decided not to follow that line of reasoning for two reasons. On the one hand a proof  already exists in \cite{RTV3} as we cited. On the other hand, in the related case of ``matrix Schubert cells'' we carry out such a program in the next section, see also \cite{FRW}.

\section{Motivic Chern classes of matrix Schubert cells}\label{sec:tauyMO} \label{sec:mCmS}

For $k\leq n$ let $\GL_k(\C)\times B_n^-$ ($B_n^-$ is the  group of $n\times n$ lower triangular non-singular matrices) act on $\Hom(\C^k,\C^n)$ by $(A,B)\cdot M=BMA^{-1}$. The orbits of this action are parameterized by $d$-element subsets $J$ of $\underline{n}=\{1,\ldots,n\}$, where $0\leq d\leq k$. We will use the notation $J=\{j_1<j_2<\ldots<j_d\}$. The orbit corresponding to $J$ is
\[
\Omega_{k,n,J}=
\{
M\in \C^{n\times k} :
\rk(\text{top $r$ rows of }M)=|J \cap \underline{r}|
\}.
\]
The motivic Chern classes of these orbits live in the $\GL_k(\C)\times B_n^-$-equivariant K-theory algebra of a point (extended by the formal variable $y$), i.e.
\begin{equation}\label{eqn:ring}
\Z[y,\alpha^{\pm 1}_1,\ldots,\alpha^{\pm 1}_k,\beta^{\pm 1}_1,\ldots,\beta^{\pm 1}_n]^{S_k},
\end{equation}
where the $S_k$-action permutes the $\alpha_u$ variables.

\subsection{Weight functions for $\Hom(\C^k,\C^n)$} \label{sec:weight}

Now we define another version of weight functions, also denoted by $W$. Since their indexing is different from weight functions of Section \ref{sec:weightfunctions}, the notational coincidence will not cause misunderstanding (c.f. Remark \ref{rem:specW}).

\begin{definition}
Let $k\leq n$, $I=\{i_1<\ldots<i_d\}\subset \underline{n}$, $|I|=d\leq k$. Let $\aalpha=(\alpha_1,\ldots,\alpha_k)$ and $\bbeta=(\beta_1,\ldots,\beta_n)$ be two sets of variables. Define
\[
W_{k,n,I}(\aalpha,\bbeta)\,=\frac1{(k-d)!}\sum_{\sigma\in S_k} U_{k,n,I}(\sigma(\aalpha),\bbeta)
\]
where
\begin{equation}\label{def:Udef}
U_{k,n,I}(\aalpha,\bbeta)=\prod_{u=1}^{k}\,\prod_{v=1}^{n}\psi_{I,u,v}(\alpha_u/\beta_v)
\cdot \prod_{u=1}^d\prod_{v=u+1}^k\frac{1+y\frac{\alpha_v}{\alpha_u}}{1-\frac{\alpha_v}{\alpha_u}},
\end{equation}
and
\[
\psi_{I,u,v}(\xi)=\left\{\begin{alignedat}4
&1- \xi&\phantom{A}&\text{ if }&\phantom{A}& u>d\;\text{ or }\;&v<i_u\,,
\\[2pt]
&(1+y)\xi &&\text{ if }&& u\leq d\;\text{ and }\; &v=i_u\,,
\\[2pt]
&1+y \xi &&\text{ if }&& u\leq d\;\text{ and }\; &v>i_u\,.
\end{alignedat}\right.
\]
\end{definition}
Although formally the weight function looks a rational function, it is in fact an element of the Laurent polynomial ring (\ref{eqn:ring}).

\begin{remark} \label{rem:specW}
The weight function of this section {\em in the special case of $d=k$} is a special case of the notion of weight function in Section \ref{sec:weightfunctions}. Namely $W_{(I,\underline{n}-I)}(\aalpha,\bbeta)$ (in the sense of Section \ref{sec:weightfunctions}) equals $W_{k,n,I}$ (in the sense of this section).
\end{remark}

\begin{example}\rm
We have
\[
W_{1,2,\{1\}}=
(1+y)\frac{\alpha_1}{\beta_1}\left( 1+\frac{y\alpha_1}{\beta_2}\right),
\qquad
W_{1,2,\{2\}}=
(1+y)\left( 1-\frac{\alpha_1}{\beta_1}\right)\frac{\alpha_1}{\beta_2},
\]
\[
W_{1,2,\{\}}=
\left( 1-\frac{\alpha_1}{\beta_1}\right)\left( 1-\frac{\alpha_1}{\beta_2}\right).
\]
More generally
\[
W_{1,n,\{u\}}=
(1+y)
\prod_{i=1}^{u-1}\left( 1-\frac{\alpha_1}{\beta_i}\right)
\cdot\frac{\alpha_1}{\beta_u}\cdot
\prod_{i=u+1}^{n}\left( 1+\frac{y\alpha_1}{\beta_i}\right),
\qquad
W_{1,n,\{\}}=
\prod_{i=1}^{n}\left( 1+\frac{y\alpha_1}{\beta_i}\right).
\]
For larger $k$ the expanded form of weight functions is less manageable, e.g.

$$W_{2,2,\{1,2\}}=(1+y)^2 \frac{\alpha _1 \alpha _2}{\beta _1 \beta _2}\cdot \left(y^2\frac{\alpha _1 \alpha _2 }{\beta _1 \beta
   _2}+y\left(-\frac{\alpha _1
   \alpha _2}{\beta _1 \beta _2}+\frac{\alpha _1}{\beta _1}+\frac{\alpha _1}{\beta
   _2}+\frac{\alpha _2}{\beta _1}+\frac{\alpha _2}{\beta
   _2}-1\right) +1\right)$$

\begin{multline*}
W_{2,4,\{2,3\}}=
(1+y)^2\prod_{i=1}^2\left(1-\frac{\alpha_i}{\beta_1}\right)\left(1+y\frac{\alpha_i}{\beta_4}\right)\cdot\alpha_1\alpha_2
 \\
\times\left(
\frac{(1-y)}{\beta_2\beta_3}
+(y^2-y)\frac{\alpha_1\alpha_2}{\beta_2^{2}\beta_3^{2}}
+y(\alpha_1+\alpha_2)\left( \frac{1}{\beta_2\beta_3^{2}} + \frac{1}{\beta_2^{2}\beta_3}\right)
\right).
\end{multline*}
\end{example}

The reader may find it instructive to verify that the sum of all weight functions for $k=1$ and fixed $n$ factors as
$\prod_{i=1}^n (1+y\alpha_1/\beta_i)$. The analogous fact for all $k,n$ will follow from Theorem \ref{thm:tauyMO}.

\subsection{Motivic Chern classes of matrix Schubert cells} \label{sec:tauyW}

Recall from the beginning of Section~\ref{sec:tauyMO} the $\GL_k(\C)\times B_n^-$-action on $\Hom(\C^k,\C^n)$ and its orbits $\Omega_{k,n,J}$.

\begin{theorem}\label{thm:tauyMO}
For the $\GL_k(\C)\times B_n^-$-equivariant motivic Chern classes we have
\[
\tauy(\Omega_{k,n,J}\subset \Hom(\C^k,\C^n))=
W_{k,n,J}.
\]
\end{theorem}

A key difference between this theorem and Theorem \ref{thm:cptmCW} is that here the named class lives in a (Laurent) polynomial ring, rather than in a quotient of such a ring by an ideal. Hence here the weigh function {\em is} the motivic Chern class, not just {\em represents} a motivic Chern class.

\begin{proof}
As before we have $J=\{j_1<\ldots<j_d\}\subset \underline{n}$ with $|J|=d$. Let $\Fl_{d,k}$ be the partial flag variety parameterizing chains of subspaces
\[
 V^\bullet=(V^0\supset V^1\supset \dots \supset V^{d}),\qquad\dim(V^i)=k-i,
\]
let $M= \Hom(\C^k,\C^n)\times\Fl_{d,k}$, and define
\[
 \tilde\Omega_{k,n,J}=
 \left\{(f,V^\bullet)\in M
 \ |\
 f(V^{u-1})\subset F^{j_u-1}, f(V^{u-1})\not\subset F^{j_u} \text{ for all } {u\in\underline{d}}
\text{ and}\;f(V^{d})=\{0\}\right\}.
\]
Consider the $(\C^*)^{k+n}$-equivariant diagram
\[
\xymatrix{\tilde{\Omega}_{k,n,J}\ar[d]^{\simeq}\ar@{^{(}->}[r]& M \ar[d]^{\rho} \ar[r]^-{\pi_2}
&\Fl_{d,k}\\
{\Omega}_{k,n,J}\ar@{^{(}->}[r]& \Hom(\C^k,\C^{n})
}
\]
with the projections $\rho=\pi_1$ and $\pi_2$. The map $\rho$  restricted to $\tilde{\Omega}_{k,n,J}$ is an isomorphism to its image $\Omega_{k,n,J}\subset \Hom(\C^k,\C^n)$ (this resolution is a version of the construction in \cite{KL}, and it also appeared in \cite{balazs-phd,FR}). Therefore, by functoriality of $\tauy$ classes, we have
\[
\tauy( \Omega_{k,n,J}\subset \Hom(\C^k,\C^n))=\rho_*( \tauy(\tilde{\Omega}_{k,n,J} \subset M)).
\]
The Lefschetz-Riemann-Roch theorem of \cite[Thm. 3.5]{Tho}, \cite[5.11.7]{ChrissGinz} or Appendix \ref{LRRgeneral} implies the following localization formula:
\begin{proposition} \label{LRR} Let $K$ be a compact smooth $\T$-variety with finitely many fixed points, $V$ a $\T$-vector space with no zero weight and denote $M:=K\times V$. Then for all $\omega\in K^\T(M)$
\[\frac{(\pi_V)_*( \omega ){|0}}{\lambda_{-1}(T^*_0V)}
=
\sum_{x\in M^{\T}}
\frac{  \omega{|x}}{\lambda_{-1}(T^*_xM)}\in Frac(K^\T(0)), \]
where $\pi_V:K\times V\to V$ is the projection and the equation is meant in the appropriate localization of $K^\T(M)$.
\end{proposition}
Applying to our situation we get
\begin{equation}\label{eqn:snail2}
\frac{\rho_*( \omega ){|0}}{\lambda_{-1}(T^*_0\Hom(\C^k,\C^n))}
=
\sum_{x\in M^{\T}}
\frac{  \omega{|x}}{\lambda_{-1}(T^*_xM)}\in Frac(K^\T(0))
\end{equation}
because the set of torus fixed points $M^{\T}$ is finite.
The formula is equivalent to the much simpler
\begin{equation}\label{eq:LLR-simplified}   \pi_* \omega|0 = \sum_{x\in K^{\T}} \frac{  \omega{|x}}{\lambda_{-1}(T^*_xK)}\in Frac(K^\T(0))[y]. \end{equation}
(Also note that the inclusions $0\to V$ and $K\to K\times V$ induce isomorphisms on $K$-theory.)
The $\T$-fixed points of $M$ are on the zero section $\Fl_{d,k}$, and they are parameterized by cosets $S_k/S_{k-d}$. We rewrite (\ref{eqn:snail2}) for $\omega=\rho_*( \tauy(\tilde{\Omega}_{k,n,J} \subset M))$ as
\begin{equation}\label{eqn:snail1}
\rho_*( \tauy(\tilde{\Omega}_{k,n,J} \subset M)){|0}
=
\sum_{x\in M^{\T}}
\frac{  \tauy(\tilde{\Omega}_{k,n,J}\subset M){|x}}{\lambda_{-1}(T^*_x\Fl_{d,k})}.
\end{equation}
One of the fixed points is $V_\text{last}=(V^\bullet)$ with $V_\text{last}^u=\spa(\epsilon_{u+1},\epsilon_{u+2},\ldots,\epsilon_k)$ (with $\epsilon_i$ being the standard basis vectors of $\C^k$). First we calculate the term on the right hand side of (\ref{eqn:snail1}) corresponding to this fixed point.

A neighborhood of $V^\bullet_\text{last}$ in $M$, ``vertically'' is simply the neighborhood of $V^\bullet_\text{last}$ in $\Fl_{d,k}$, which is naturally identified with the set of unipotent $k\times k$ matrices illustrated in the first picture below. The ``horizontal'' coordinates are the entries of an $n\times k$ matrix.
\begin{center}
\begin{tikzpicture}[scale=.75]
\draw(-1,0) -- (5,0) -- (5,6) -- (-1,6) -- (-1,0);
\draw (-1,2) -- (5,2);
\draw (3,0) -- (3,6);
\draw[fill=gray, opacity=.4]
(-.5,6) --(-.5,5.5) --
(0,5.5)--(0,5)--
(0.5,5)--(0.5,4.5) --
(1,4.5)--(1,4) --
(1.5,4)--(1.5,3.5) --
(2,3.5)--(2,3) --
(2.5,3)--(2.5,2.5) --
(3,2.5) -- (3,6) -- (-.5,6);
\draw (-.75,5.75) node {1};
\draw (-.25,5.25) node {1};
\draw (.25,4.75) node {1};
\draw (.75,4.25) node {1};
\draw (1.25,3.75) node {1};
\draw (1.75,3.25) node {1};
\draw (2.25,2.75) node {1};
\draw (2.75,2.25) node {1};
\draw (3.25,1.75) node {1};
\draw (3.75,1.25) node {1};
\draw (4.25,0.75) node {1};
\draw (4.75,0.25) node {1};
\draw[fill=gray, opacity=.4] (3,2) -- (5,2) -- (5,6) -- (3,6);

\draw (2,-1.25) node {A neighborhood of $V^\bullet_\text{last}$ in $\Fl_{d,k}$};
\draw (5.25,4) node {$d$};
\draw (5.75,1) node {$k-d$};
\draw (-1.5,5.75) node {$\alpha_1$};
\draw (-1.5,4.2) node {$\vdots$};
\draw (-1.5,2.35) node {$\alpha_d$};
\draw (-1.75,1.65) node {$\alpha_{d+1}$};
\draw (-1.5, 1.05) node {$\vdots$};
\draw (-1.5, .25) node {$\alpha_k$};
\draw (-.7,6.4) node {\footnotesize$\alpha_1$};
\draw (1,6.4) node {\footnotesize$\cdots$};
\draw (2.5,6.4) node {\footnotesize$\alpha_d$};
\draw (3.55,6.4) node {\footnotesize$\alpha_{d+1}$};
\draw (4.35, 6.4) node {\footnotesize$\cdots$};
\draw (4.9, 6.4) node {\footnotesize$\alpha_k$};

\draw (8,-1.5) -- (14,-1.5) -- (14,7.5) -- (8,7.5) -- (8,-1.5);
\draw (12,-1.5) -- (12,7.5);
\draw[fill=gray, opacity=.4] (8,6.5) --(8.5,6.5)
--(8.5,5.5)-- (9,5.5)
-- (9,4.5) -- (9.5,4.5)
-- (9.5,3.5) -- (10,3.5)
-- (10,2.5) -- (10.5,2.5)
-- (10.5,1.5) -- (11,1.5)
--(11,0.5) --(11.5,0.5) --(11.5,-0.5) -- (12,-0.5)
-- (12,-1.5) -- (8,-1.5)
;
\draw (8.25,6.25) node {$\bullet$};
\draw (8.75,5.25) node {$\bullet$};
\draw (9.25,4.25) node {$\bullet$};
\draw (9.75,3.25) node {$\bullet$};
\draw (10.25,2.25) node {$\bullet$};
\draw (10.75,1.25) node {$\bullet$};
\draw (11.25,0.25) node {$\bullet$};
\draw (11.75,-.75) node {$\bullet$};
\draw (10.5,5.5) node {\Large$0$};
\draw (13,3) node {\Large$0$};
\draw (9.25,.25) node {\Huge$*$};

\draw (7.6,6.3) node {\footnotesize$j_1$};
\draw (7.6,5.3) node {\footnotesize$j_2$};
\draw (7.6,4.3) node {\footnotesize$j_3$};
\draw (7.6,3.3) node {\footnotesize$j_4$};
\draw (7.6,2.3) node {\footnotesize$j_5$};
\draw (7.6,1) node {\footnotesize$\vdots$};
\draw (7.6,-0.8) node {\footnotesize$j_d$};
\draw (8.2,7.9) node {\footnotesize$\alpha_1$};
\draw (10,7.9) node {\footnotesize$\cdots$};
\draw (11.55,7.9) node {\footnotesize$\alpha_d$};
\draw (12.55,7.9) node {\footnotesize$\alpha_{d+1}$};
\draw (13.35, 7.9) node {\footnotesize$\cdots$};
\draw (13.9, 7.9) node {\footnotesize$\alpha_k$};

\draw (14.5,7.25) node {\footnotesize$\beta_1$};
\draw (14.5,6.75) node {\footnotesize$\beta_2$};
\draw (14.5,3) node {\footnotesize$\vdots$};
\draw (14.6,-.75) node {\footnotesize$\beta_{n-1}$};
\draw (14.5,-1.25) node {\footnotesize$\beta_n$};

\end{tikzpicture}
\end{center}
Consider the subset of $n \times k$ matrices illustrated in the second figure, that is, the entries of the matrix labeled with 0 are 0, the entries labeled by $\bullet$ are non-zero complex numbers, and the entries labeled by $*$ are arbitrary complex numbers. Easy calculation from the definition of $\tilde{\Omega}_{k,n,J}$ shows that the product of matrices should belong to the subset of $\Hom(\C^k,\C^n)$ described above.
New coordinates are obtained by composing matrices (see Example \ref{ex:spec} bellow).

Using the fundamental calculations of Section \ref{sec:fundcalc} we obtain that the numerator of the term corresponding to $V^\bullet_\text{last}$ is the product of factors
\begin{itemize}
\item $1+y\alpha_v/\alpha_u$ for $1\leq u\leq d$, $u+1\leq v\leq k$ (coming from vertical directions, see the first picture);
\item $1+y\alpha_u/\beta_v$ for $1\leq u\leq d$, $j_u< v\leq n$ (the *-entries of the second picture);
\item $1-\alpha_u/\beta_v$ for $1\leq u\leq d$ and $1\leq v<j_u$, or $d<u\leq k$, $1\leq v\leq n$ (the 0-entries of the second picture);
\item $(1+y)\alpha_u/\beta_{j_u}$ for $1\leq u\leq d$ (the $\bullet$-entries in the second picture);
\end{itemize}
while the denominator is $\prod_{u=1}^d\prod_{v=1}^k (1-\alpha_v/\alpha_u)$.

That is, the term on the right hand side of (\ref{eqn:snail1}) corresponding to the fixed point $V^\bullet_\text{last}$ is exactly $U_{k,n,J}$ of (\ref{def:Udef}).
The terms corresponding to the other fixed points, i.e. corresponding to other cosets of $S_k/S_{k-d}$, are obtained by permuting the $\alpha$-variables. Therefore the right hand side of (\ref{eqn:snail1}) is $1/(k-d)!\cdot U_{k,n,J}(\sigma(\aalpha))=W_{k,n,J}$ what we wanted to prove.
\end{proof}

\begin{example} \label{ex:spec} \rm
Here we show the details of the coordinate change used in the proof above for the concrete example $k=4$, $n=4$, $J=\{2,3\}$ ($d=2$).
The neighbourhood of the standard flag $V^\bullet_\text{last}$,  in the flag variety $\Fl_{2,4}$ is parameterized by the matrices of the form
$$\begin{matrix}\left(
\begin{array}{cccc}
 \bullet & 0 &  0&  0\\
 * & \bullet & 0&  0 \\
 * & * & * &  *\\
 * & * & * &  *\\
\end{array}\right)&\phantom{cccccccccccc}&\left(
\begin{array}{cccc}
 1 & * & * & *\\
 0 & 1 & * & *\\
 0 & 0 & 1 & 0\\
 0 & 0 & 0 & 1\\
\end{array}\right)\\ \\
\text{stabilizer of }(V^1_\text{last},V^2_\text{last}),&&\text{local coordinates.}\\
\end{matrix}$$
The set $\tilde\Omega_{2,4,\{2,3\}}$  is described by imposing conditions on the entries of the product of the matrices:
\begin{multline*}\left(
\begin{array}{cccc}
 f_{11} & f_{12} & f_{13}& f_{14} \\
 f_{21} & f_{22} & f_{23}& f_{24} \\
 f_{31} & f_{32} & f_{33}& f_{34} \\
 f_{41} & f_{42} & f_{43}& f_{44} \\
\end{array}
\right)\cdot \left(
\begin{array}{cccc}
 1 & g_{12} & g_{13}& g_{14} \\
 0 & 1 & g_{23} & g_{24}\\
 0 & 0 & 1 & 0\\
 0 & 0 & 0 & 1\\
\end{array}
\right)=\\=
\left(
\begin{array}{cccc}
 f_{11} & f_{11}g_{12}+f_{12} & f_{11}g_{13}+f_{12}g_{23}+f_{13}& f_{11}g_{14}+f_{12}g_{24}+f_{14}\\
\boxed{\bm {f_{21}}} & f_{21}g_{12}+f_{22} & f_{21}g_{13}+f_{22}g_{23}+f_{23}& f_{21}g_{14}+f_{22}g_{24}+f_{24}\\
\bm{ f_{31}} & \boxed{\bm {f_{31}g_{12}+f_{32}}} & f_{31}g_{13}+f_{32}g_{23}+f_{33}& f_{31}g_{14}+f_{32}g_{24}+f_{34}\\
\bm{f_{41}} & \bm{f_{41}g_{12}+f_{42}} & f_{41}g_{13}+f_{42}g_{23}+f_{43}& f_{41}g_{14}+f_{42}g_{24}+f_{44}\\
\end{array}
\right).\end{multline*}
\begin{itemize}
\item the boxed entries should be nonzero,
\item the bold entries are arbitrary,
\item the remaining entries have to vanish.
\end{itemize}
Let us introduce new coordinates  $\{f'_{st}\}_{ s\in\underline{n},\;t\in\underline{k}}$,
which coincide with the  entries of the product matrix. Together with $\{g_{uv}\}_{1\leq u\leq d,\; u<v\leq k}$ it is a coordinate system in a neighborhood of the fixed point $(0,V^\bullet_\text{last})\in M$, which lies in the closure of $\tilde\Omega_{2,4,\{2,3\}}$. In the new coordinates  the set $\tilde \Omega_{2,4,\{2,3\}}$ is defined by the conditions:
$$
\begin{array}{llll}
 f'_{11}=0, & f'_{12}=0, & f'_{13}=0,& f'_{14}=0, \\
 \boxed{\bm{f'_{21}}\not=0,} & f'_{22}=0, & f'_{23}=0,& f'_{24}=0,\\
& \boxed{\bm{f'_{32}}\not=0}, &f'_{33}=0,&f'_{34}=0, \\
&                        &f'_{43}=0. &f'_{44}=0. \\
\end{array}
$$
\end{example}

\begin{remark}
A sketch of the cohomology version of the proof above appeared in \cite[Sect.~12]{FR}. Not surprisingly the cohomological weight function (CSM-class) only differs from the one in this paper by carrying out the changes
\[
1-\frac{1}{\alpha} \mapsto a, \qquad 1+\frac{y}{\alpha} \mapsto 1+a, \qquad \frac{1+y}{\alpha} \mapsto 1
\]
for each of its factors, c.f. Remark \ref{rem:coh}.
\end{remark}

\section{$A_2$ quiver representation or determinantal varieties} \label{sec:A2}

Chern-Schwartz-MacPherson classes of determinantal varieties for $A_2$ quivers were calculated in \cite{PP}. A slightly different proof was given in \cite{FR}. We will adapt the latter proof to calculate the motivic Chern classes.

\subsection{Segre version of motivic Chern class} Imitating the way the Segre-Schwartz-MacPherson class is defined from the Chern-Schwartz-MacPherson class (namely SSM=CSM/total-Chern-class of ambient space), we define the Segre version of the motivic Chern class.

\begin{definition}
For $X\to V$ ($V$ smooth) define the motivic Segre class as
\[
\ts(X\to V)=\tauy(X\to V)/\lambda_y(T^*V)\in  K(V)[[y]].
\]
We use the abbreviation $\ts(X)$ if the map $X\to V$ is an embedding and clear from the context, and we use the same notation for the equivariant version (living in $K^G(V)[[y]]$) if the map $X\to V$ is invariant under a group action.
\end{definition}

The motivic Segre class contains the same information as the motivic Chern class, but some formulas are more transparent in terms of Segre classes, in particular we expect certain stabilization properties of the motivic Segre class, similarly to Segre-Schwartz-MacPherson classes of determinantal varieties.

For example suppose that the embedded smooth variety $\iota:X\hookrightarrow V$ is the zero locus of a section $\sigma: V\to \xi$ of a vector bundle $\xi$ over $V$ such that $\sigma$ is transversal to the zero section. Then  transversality and the multiplicativity of $\lambda_y$  implies

\begin{equation}\label{segre-of-locus}\ts(\iota)=
\frac {\lambda_{-1}(\xi^*)}{\lambda_{y}(\xi^*)}\,.\end{equation}

\subsection{Fixed rank loci}
For $k\leq n$ let $G\!:=\GL_k(\C)\times \GL_n(\C)$ act on $V\!:=\Hom(\C^k,\C^n)$ by $(A,B)\cdot M=BMA^{-1}$. The orbits of this action are characterized by rank, namely
\[
\Sigma^r_{k,n}=
\{\phi\in \Hom(\C^k,\C^n): \dim \ker \phi=r\}
\]
are the orbits for $r=0,\ldots,k$. We will be interested in the $G$-equivariant motivic Chern/Segre classes of these orbits living in
\[
K^\T(\Hom(\C^k,\C^n))[[y]]= K^\T(\pt)[[y]]=\Z[\alpha^{\pm 1}_1,\ldots,\alpha^{\pm 1}_k,\beta^{\pm 1}_1,\ldots,\beta^{\pm 1}_n][[y]]
\]
where $\T=(\C^*)^{k}\times (\C^*)^{n}$, c.f. Section \ref{sec:mCmS}. The first equation is a consequence of the fact that the equivariant $K$-theory of a vector space is canonically isomorphic to the equivariant $K$-theory of a point. The answers will be symmetric in the $\alpha_i$ and $\beta_j$ variables separately (since the orbits are obviously $\GL_k(\C)\times \GL_n(\C)$-invariant), and can be considered as the $\GL_k(\C)\times \GL_n(\C)$-equivariant $\tauy$ and $\ts$ classes.

In the next two sections we will present two formulas for $\tauy$ (or $\ts$) of $\Sigma^r_{k,n}$. The two expressions obtained in Theorems \ref{thm:tauA2v1}, \ref{thm:tauA2v2} will be rather different, their equality is not clear algebraically.

\subsection{Motivic Chern classes of fixed rank loci --- the motivic expression}

Recall the definition of weight functions $W_{k,n,J}$ from Section \ref{sec:weight}.

\begin{theorem}\label{thm:tauA2v1}
We have
\begin{align*}
\tauy(\Sigma^r_{k,n}) =&
\mathop{\sum_{J\subset \underline{n}}}_{|J|=k-r}
W_{k,n,J}(\aalpha,\bbeta), \\
\ts(\Sigma^r_{k,n}) =&
\frac{1}{\prod_{i=1}^k\prod_{j=1}^n \left( 1+ y\alpha_i/\beta_j\right)}\cdot
\mathop{\sum_{J\subset \underline{n}}}_{|J|=k-r}
W_{k,n,J}(\aalpha,\bbeta).
\end{align*}
\end{theorem}

\begin{proof}
The set $\Sigma^r_{k,n}$ is the disjoint union of the sets $\Omega_{k,n,J}$ with
$J\subset \underline{n}$, $|J|=k-r$. This decomposition is $\T=(\C^*)^k\times (\C^*)^n$-invariant. Hence the motivic properties of $\T$-equivariant $\tauy$ and $\ts$ classes imply the statements.
\end{proof}

\subsection{Motivic Chern classes of fixed rank loci --- the sieve expression} \label{sec:sieve}

In this section it will be convenient to denote $q=-y$.
Let
$${a\choose r}_{\!\!q}=\chi_{-q}(\Gr_r(\C^a))=|\Gr_r(\mathbb F_q^a)|\,.$$
be the $q$-binomial coefficient satisfying the recursion
\[
{a+1\choose r}_{\!\!q}=q^r{a\choose r}_{\!\!q}+{a\choose r-1}_{\!\!q},\qquad
{a\choose 0}_{\!\!q}={a\choose a}_{\!\!q}=1.
\]

\begin{theorem}\label{thm:tauA2v2}
We have
\begin{equation}\label{eqn:recsolved}
\ts(\Sigma^r_{k,n})=\sum_{a=r}^k(-1)^{a-r}q^{\frac{(a-r)(a-r-1)}2}{a\choose r}_{\!\!q}\Phi^a_{k,n},
\end{equation}
where
\begin{equation}\label{eqn:Phiexplicit}
\Phi^a_{k,n}=\sum_{I\in {k\choose a}}\prod_{u\in I}\prod_{v=1}^n\frac{1-\tfrac{\alpha_u}{\beta_v}}{1-q\tfrac{\alpha_u}{\beta_v}}
\prod_{u\in I}\prod_{w\in I^\vee}\frac{1-q\tfrac{\alpha_u}{\alpha_w}}{1-\tfrac{\alpha_u}{\alpha_w}}.
\end{equation}
(Here by $I\in {k \choose a}$ we mean that $I$ is a subset of $\underline{k}$ of cardinality $a$, and $I^\vee$ is its complement.)
\end{theorem}

\begin{proof}

Let $\Gr_{a}(\C^k)$ denote the Grassmannian of $a$-dimensional subspaces in $\C^k$. Consider the smooth variety
\[
\widetilde \Sigma^a_{k,n}=\{(W,f)\in \Gr_{a}(\C^k)\times V\;|\;W\subset \ker(f)\}.
\]
$k$ and $n$ will be fixed, so occasionally we drop them from the notation. Let $\eta_a:\widetilde \Sigma^a_{k,n}\to V$ be the restriction of the projection $\pi_V:K\times V\to V$ for $K=\Gr_{a}(\C^k)$; and denote
$\Phi^a_{k,n}=\ts(\eta_a)$. Notice that $\eta_a$ is a resolution of $\overline\Sigma^a_{k,n}$, however it has not the normal crossing property. The set $\overline\Sigma^a_{k,n}$ can be partitioned according to the corank:
\[ \overline\Sigma^a_{k,n}=\coprod_{r=a}^{k} \Sigma^r_{k,n}.\]
Using the notation $X_{k,n}^{r,a}:=\eta_a^{-1}(\Sigma^r_{k,n})$ and the motivic property we have
\[ \tauy(\eta_a)=\sum_{r=a}^{k}\tauy(\eta_a|_{X_{k,n}^{r,a}}).\]
Notice now that $\eta_a|_{X_{k,n}^{r,a}}$, as a map to its image $\Sigma^r_{k,n}\subset V$, is a fiber bundle with fiber $\Gr_r(\C^a)$, equivariantly locally trivial in Zariski topology.

An important feature of both $\tauy$ and $\ts$ is that they are not necessarily multiplicative with respect to fibrations. Nevertheless the formula
\begin{equation}\label{eqn:Zariski}
p_*(\tauy(X\to Y))=\chi_y(F)\tauy(id_Y)
\end{equation}
holds for fibrations $F\hookrightarrow X\stackrel p\twoheadrightarrow Y$ which are {\it equivariantly locally trivial in Zariski topology}.
This holds because for such a fibration one can decompose the base into invariant locally closed sets over which the bundle is of the product form, see \cite[Rem.~4.3]{AMSS2}.
\medskip

Nonequivariantly (in Zariski topology) the map $\eta_{|X_{k,n}^{r,a}}:X_{k,n}^{r,a}\to\Sigma_{k,n}^r$ is locally trivial. Indeed, it is the Grassmann bundle associated to the vector bundle of kernels
$$Ker_{k,n}^{r}=\{(v,f)\in \C^k\times \Sigma^r_{k,n}\;|\;f(v)=0\;\}\,,$$
$$X_{k,n}^{r,a}=\Gr_a(Ker_{k,n}^{r})\,.$$
Every vector bundle is
locally trivial in Zariski topology. Also we can write an
explicit $\T$ equivariant trivialization. We choose some subset $J$ of coordinates of cardinality $r$. The projection along coordinates $Ker_{k,n}^{r}\to \Sigma^r_{k,n}\times\C^{r}_J$ is $\T$ equivariant, and over an open subset $U_J\subset \Sigma^r_{k,n}$ it is an isomorphism. The sets $U_J$ cover $\Sigma^r_{k,n}$. Applying the Grassmann construction we obtain equivariant trivialization of the bundle $X_{k,n}^{r,a}\to\Sigma^r_{k,n}$.

This implies that
\[\tauy(\eta_a|_{X_r})=\chi_y(\Gr_r(\C^a))\tauy(\Sigma^r),\]
where, according to our convention,  $\tauy(\Sigma^r)=\tauy(\Sigma^r_{k,n}\subset V)$.
Consequently
\[ \tauy(\eta_a)=\sum_{r=a}^{k}\chi_y(\Gr_r(\C^a))\tauy(\Sigma^r)=\sum_{r=a}^{k}\binom{a}{r}_{\!\!q}\tauy(\Sigma^r).\]
Dividing by $\tauy(V)$ we get an equivalent
\[ \Phi^a=\ts(\eta_a)=\sum_{r=a}^{k}\binom{a}{r}_{\!\!q}\ts(\Sigma^r).\]

Using the elementary fact that the matrices
\[
\left[{a\choose r}_{\!\!q}\right]_{1\leq a,r\leq k},\qquad\qquad
\left[ (-1)^{a-r}q^{\frac{(a-r)(a-r-1)}{2}} {a\choose r}_{\!\!q}\right]_{1\leq a,r\leq k}
\]
are inverses to each other we verified \eqref{eqn:recsolved}.

To calculate $\Phi^a_{k,n}=\ts(\eta_a)$ notice that $\eta_a=\pi_V\iota_a$ for the embedding $\iota_a:\widetilde \Sigma^a_{k,n}\to K\times V$ implying that $\tauy(\eta_a)=(\pi_V)_*\tauy(\iota_a)$. In terms of Segre classes this means $$\ts(\eta_a)=(\pi_V)_*\big(\ts(\iota_a)\,\pi_K^*(\lambda_y(T^*K))\big)\,.$$

Consider the  vector bundle $\xi=\Hom(\gamma_a,\C^n)\to K\times V$, where $\gamma_a$ is the tautological bundle over $\Gr_a(\C^k)$. It has a section $\sigma$ defined by $\sigma(W,f):=f|W$. The zero locus of $\sigma$ is $\widetilde \Sigma^a_{k,n}$ by definition and it is easy to check that $\sigma$ is transversal to the zero section of $\xi$.
Therefore we can apply \eqref{segre-of-locus} to obtain
\[\ts(\iota_a)=\frac {\lambda_{-1}(\xi^*)}{\lambda_{y}(\xi^*)}.\]
We can now apply the Lefschetz-Riemann-Roch formula in the form \eqref{eq:LLR-simplified}. The summands in the expression for the motivic Chern class are of the form
$$
\frac{{\lambda_{-1}(\xi^*)} }{\lambda_{-q}(\xi^*)}\cdot \frac{\lambda_{-q}(T_x^*K)}{\lambda_{-1}(T_x^*K)}\,.$$
The factor
$$\prod_{u\in I}\prod_{v=1}^n\frac{1-\tfrac{\alpha_u}{\beta_v}}{1-q\tfrac{\alpha_u}{\beta_v}}$$
in \eqref{eqn:Phiexplicit} corresponds to $$\frac{\lambda_{-1}\xi^*}{\lambda_{-q}\xi^*}=\frac{\lambda_{-1}(\Hom(\gamma_a,\C^n)^*)}{\lambda_{-q}(\Hom(\gamma_a,\C^n)^*)}$$
and
$$\prod_{u\in I}\prod_{w\in I^\vee}\frac{1-q\tfrac{\alpha_u}{\alpha_w}}{1-\tfrac{\alpha_u}{\alpha_w}}$$
corresponds to
$$\frac{\lambda_{-q}(T_x^*K)}{\lambda_{-1}(T_x^*K)}=\frac{\lambda_{-q}(\Hom(\gamma_a,\gamma_a^\perp)^*)}{\lambda_{-1}(\Hom(\gamma_a,\gamma_a^\perp)^*)}\,.$$
Writing the above expression in coordinate roots
  we arrive to the conclusion.

\medskip
\end{proof}

\begin{remark} \rm
It would be interesting to understand the $k,n$-dependence of the formulas in Theorem \ref{thm:tauA2v2}, either combinatorially, or via the language of iterated residues, cf. \cite{FR}. For now, what is clear from the formulas is that both $\Phi^r_{k,n}$ and $\ts(\Sigma^r_{k,n})$ are supersymmetric, i.e. satisfy the functional equation
\[
f(\alpha_1,\ldots,\alpha_k,t;\beta_1,\ldots,\beta_n,t)=
f(\alpha_1,\ldots,\alpha_k;\beta_1,\ldots,\beta_n),
\]
with base case
\[
\ts(\Sigma^k_{k,n})=\ts(\{0\}\subset V)=\Phi^k_{k,n}=\prod_{u=1}^k\prod_{v=1}^n\frac{1-\tfrac{\alpha_u}{\beta_v}}{1-q\tfrac{\alpha_u}{\beta_v}}.
\]
\end{remark}

\subsection{Examples}
For two of the rank loci of the representation $\Hom(\C^2,\C^2)$ we have
\[
\ts(\Sigma^2_{2,2})
=
\frac{\left(1-\frac{\alpha_1}{\beta_1}\right)
   \left(1-\frac{\alpha_2}{\beta_1}\right)
   \left(1-\frac{\alpha_1}{\beta_2}\right)
   \left(1-\frac{\alpha_2}{\beta_2}\right)}
  {\left(1-q\frac{\alpha_1}{\beta_1}\right)
   \left(1-q\frac{\alpha_2}{\beta_1}\right)
   \left(1-q\frac{\alpha_1}{\beta_2}\right)
   \left(1-q\frac{\alpha_2}{\beta_2}\right)},
\]

\begin{multline*}
\ts(\Sigma^0_{2,2})=
\frac{(q-1)^2}{\left(1-q\frac{\alpha_1}{\beta_1}\right)
   \left(1-q\frac{\alpha_2}{\beta_1}\right)
   \left(1-q\frac{\alpha_1}{\beta_2}\right)
   \left(1-q\frac{\alpha_2}{\beta_2}\right)}\cdot\frac{\alpha_1 \alpha_2}{\beta_1
   \beta_2}\cdot \\
\Bigg(q^2\frac{ \alpha_1 \alpha_2}{\beta_1 \beta_2}+q
   \left(\frac{\alpha_2
   \alpha_1}{\beta_1
   \beta_2}-\frac{\alpha_1}{\beta_1}-\frac{\alpha_1}{\beta_2}-\frac{\alpha_2}{\beta_1}-\frac{\alpha_2}{\beta_2}+1\right)+1\Bigg).
\end{multline*}

The motivic Segre classes can be expressed in terms of double stable Grothendieck polynomials $G_\lambda=G_\lambda(\alpha_1,\ldots,\alpha_k;\beta_1,\ldots,\beta_n)$ , see e.g. \cite[Definition 4.2 and Theorem 4.5]{RS}. For example, for $k=n=2$ we have
\begin{align*}
\ts(\Sigma^0_{2,2})&=(G_0 - G_1)+ q (-G_1 + G_2 + G_{11} - G_{21})+
\\ q^2 &(-G_1 + 2 G_2 - G_3 + 2 G_{11} - 4 G_{21} - G_{111} + 2 G_{31} + 2 G_{211} - G_{311})+\ldots
\end{align*}
where experts on algebraic combinatorics can see the usual sign pattern of Grothendieck expansions (e.g. \cite{B}). Further analysis along these lines is subject to future research.

\section{Appendix: comparison of algebraic and topological K-theories}

\label{topKtheory}
\subsection{Algebraic K theory}
Our main reference to algebraic K theory is the chapter 5 of \cite{ChrissGinz} or original papers of Thomason, e.g. \cite{Tho}. We assume that a linear algebraic group $G$ acts on a smooth algebraic variety $X$. We restrict our attention to the case of complex base field. Usually the variety is assumed to be quasiprojective, but some technical issues can be overcome by analyzing of Thomason arguments, see \cite[p.28]{Edi}. The $K^0_G(X)=K_{alg}^G(X)$ is the  K theory of the exact category of the locally free $G$-sheaves on $X$. The higher K-groups $K^i_G(X)$ are defined by means of the Quillen Q-construction. They are functorial with respect to the maps in the category of smooth algebraic varieties.
The following properties are crucial for us:
\begin{description}
\item [A1] $K^0_G (X)=K^G _{alg}(X)$.
\item [A2] Let $X$ be a smooth $G$-variety and $Y$ a closed smooth subvariety. Then there is a functorial exact sequence
\begin{equation*}
\xymatrix{\ar[r]&K_{alg}^G(Y)\ar[r] &K_{alg}^G(X)\ar[r]&K_{alg}^G(X\setminus Y)\ar[r]&0}\\
\end{equation*}
extended by the sequence of higher K theory,
 \cite[(0.3.1)]{Tho}, \cite[5.2.14]{ChrissGinz}.
\item [A3] For a subgroup $H\subset G$ we have $K_G ^*(G \times_H X) \simeq K_H ^*(X)$ for every $H $-variety $X$, \cite[5.2.17]{ChrissGinz}.
\item[A4] $\mathbb A^1$ homotopy invariance:
 If $f:X\to Y$ is a $G$-equivariant affine bundle, then $f^*$ induces an isomorphism of K theory, \cite[5.4.17]{ChrissGinz}.
\item [A5]
If $H$ is a subgroup of $G$, then $$K^*_G(G/H)\simeq K^*_H(\pt)\,.$$
Let $L$ be a maximal reductive subgroup of $H$ such that $H=H_u\rtimes L$, where $H_u$ is the unipotent radical of $H$. Existence of such a subgroup $L$ is guaranteed by \cite[VIII, Th.~4.3]{Hoch}\footnote{If $G$ is connected then $L$ usually is called the Levi factor, otherwise we say that $L$ is a maximal reductive subgroup of $H$.}.
Sine  $H/L$ is an affine space, by the $\mathbb A^1$ homotopy invariance
$$K^*_H(\pt)\simeq K^*_H(H/L)\simeq K^*_L(\pt)\,.$$
The category of $L$-bundles over a point is semisimple, thus
$$K^0_G(G/H)\simeq K^0_L(\pt)= \Rep(L)\,.$$
The higher algebraic K-theory is much complicated. For example for the trivial group $G$
the higher K theory is divisible \cite[VI.1.6]{Weibel}. It follows
from \cite[IV.6.11]{Weibel} that $K^L_{>0}(\pt)$ is a divisible group as well\footnote{We we do not need to use this property for further arguments}.
\item [A6] If $H\subset G$ is connected, then
$$K^0_G(G/H)\simeq\Rep(L)\simeq \Rep(\T_L)^{W_L}\,,$$
where $\T_L$ is the maximal torus of $L$ and $W_L$ is the Weyl group.
\item [A7] The proper map $f:X\to Y$ induces the  push-forward of K theory
$$f_*([\F])=\sum_{i=0}^{\dim Y}(-1)^i[Rf_i(\F)]\,.$$
It is essential that the space $Y$ is smooth, since we have to replace the coherent sheaves by their equivariant resolutions. Existence of a resolution is proven by
\cite[Prop.~2.1]{Nielsen} for torus or \cite[Prop.5.1.28]{ChrissGinz}, \cite{Josh} for general $G$ acting on quasiprojective variety.  The case of varieties which are not quasi-projective is discussed in \cite[p.28]{Edi}.
Therefore we have well defined push-forward  to the equivariant K-theory of a smooth \quapr variety.
For inclusions $i:X\hookrightarrow Y$ the composition $i^*i_*$ is the multiplication by $\lambda_{-1}(\nu^*_{Y/X})$, where $\nu_{Y/X}$ is the normal bundle,
\cite[5.4.10]{ChrissGinz}.

\item [A8] (Localization theorem.) Suppose that $G =\T $ is a torus.
Then the kernel and the cokernel of the restriction to the fixed point set
$$K_\T ^*(X)\to K_\T ^*(X^\T )$$
are torsion $\Rep(\T )$-modules, \cite[Th.~2.1]{Tho}.

\end{description}

\subsection{Topological K theory}
Let $\G $ be a compact Lie group and
let $X$ be a compact topological $\G $-space, $\G$-homotopy equivalent to a $\G$--CW-complex. By $K^\G _{top}(X)$ we denote the Grothendieck ring of topological $\G $-vector bundles over $X$. It is a contravariant functor on the category of $\G $-spaces. According to  Segal \cite{Seg} this functor extends to a $\Z_2$ graded multiplicative generalized equivariant cohomology theory $K^*_\G (-)$ defined for spaces
which are of $\G $--homotopy type of finite $\G $-CW-complexes. We list below the properties corresponding to the analogous properties of algebraic K theory:
\begin{description}
\item [T1] $K^0_\G (X)=K^\G _{top}(X)$.
\item [T2] $K^*_\G$ is a $\Z_2$-graded generalized cohomology theory, in particular there is a six-term exact sequence for a pair $(X,X\setminus Y)$ (we assume that this pair is homotopy equivalent to a $\G$-CW-pair)
\begin{equation*}
\xymatrix{
&K^0_\G(X,X\setminus Y)\ar[r] &K^0_\G(X)\ar[r]&K^0_\G(X\setminus Y)\ar[d]\\
&K^1_\G(X\setminus Y)\ar[u] &K^1_\G(X)\ar[l]&K^1_\G(X,X\setminus Y).\ar[l]}
\end{equation*}
If $X$ and $Y$ are smooth manifolds and the normal bundle of $Y$ in $X$ has a $\G$--invariant almost complex structure, then $K^i_\G(X,X\setminus Y)\simeq K_\G^{i+\codim Y}(Y)$.
\item [T3]  $K_\G ^*(\G \times_\HH X) \simeq K_\HH^*(X)$ for every $\HH$-space $X$, \cite[\S2(iii)]{Seg}.
\item [T4] $K^*_\G(-)$ is  $\G$-homotopy invariant.
\item [T5] $K^0_\G (\G /\HH)\simeq \Rep(\HH)$, $K^1_\G (\G /\HH)=0$.
\item [T6] If $\G$ is a connected compact group, then
$$K_\G^0(\G)=\Rep(\G)\simeq\Rep(\Tc)^W\,,$$ where  $\Tc \subset \G $ the is the maximal compact torus 
 and  $W$ is the Weyl group of $\G$.
\item [T7] (Thom isomorphism) There is a natural isomorphism of reduced K-theories: if $E\to X$ is an equivariant complex vector bundle, then $$\tilde K_\G ^*(Th(E))\simeq \tilde K^*_\G (X)\,,$$ where $Th(E)$ is the Thom space of $E$, \cite[Prop.~3.2]{Seg}.
It allows to define push-forward for proper maps of complex $\G $-manifolds. See also \cite{Josh}.
\item [T8] (Localization theorem.) Suppose that $\G =\Tc $ is a compact torus.
Then the kernel and the cokernel of the restriction to the fixed point set
$$K_\Tc ^*(X)\to K_\Tc ^*(X^\Tc )$$
are torsion $\Rep(\Tc )$-modules, \cite[Prop.~4.1]{Seg}.
\end{description}

\subsection{Comparison}
Suppose that $L$ is a maximal reductive subgroup of a linear algebraic group $G$ and $\G\subset L$ is a maximal compact subgroup. Then	
$$K^G_{alg}(\pt)\simeq \Rep(L)\simeq \Rep(\G)\simeq K^\G_{top}(\pt)\,.$$
The isomorphism is induced by the forgetful functor.
Similarly, let
 $H$ be an algebraic subgroup of $G$, $L$ the maximal reductive subgroup of $H$ and $\HH\subset L$ its maximal compact subgroup. We can assume that $\HH\subset \G$, then $G /H$ is $\G$--homotopy equivalent to $\G/\HH$.
The higher algebraic K theory and odd topological K-theory vanishes
and $K^0_{\G}(G/H)\simeq K^0_{\G}(\G/\HH)\simeq \Rep(\HH)$. We will generalize this fact to the situation we are interested in:

\begin{theorem} \label{K-LES}
If a smooth $G$-variety $V$ is a finite union of orbits, then  $K^1_\G(V)=0$ and the natural map
$$K^G_{alg}(V)\to K^\G_{top}(V)$$ is an isomorphism.
\end{theorem}
\begin{proof}
We proceed as in the proof of Theorem \ref{thm:axiomatic}. We fix a linear order of orbits and prove theorem inductively with respect to $\Theta$. At each step we have the long exact sequences (A2) of the pair $(V_{\succeq\Theta},V_{\succ\Theta})$
\begin{equation}\label{algKex}
\xymatrix{{}\ar[r]&K^G_{alg}(\Theta)\ar[r]^{(\iota^{alg}_\Theta)_*} &K^G_{alg}(V_{\succeq\Theta})\ar[r]&K^G_{alg}(V_{\succ\Theta})\ar[r]&0}\\
\end{equation}
and the corresponding sequence (T2) of topological K theory reduces to the short one, by (T5).
We obtain
\begin{equation}
\xymatrix{0\ar[r]&K^\G_{top}(\Theta)\ar[r]^{(\iota^{top}_\Theta)_*} &K^\G_{top}(V_{\succeq\Theta})\ar[r]&K^\G_{top}(V_{\succ\Theta})\ar[r]&0.}\\
\end{equation}
To compare $K^G_{alg}(V_{\succeq\Theta})$ with $K^\G_{top}(V_{\succeq\Theta})$ we apply the natural transformation of theories, which obviously commutes with the restriction to the open subset. The push forward $(\iota^{top}_\Theta)_*$  is defined in \cite{Seg} via the Koszul complex of the normal bundle $\Lambda^\bullet \nu(\Theta)$, while in algebraic K theory the Koszul complex of the conormal bundle is needed. We obtain commutativity by precomposing the natural map with the duality in the topological K theory $E\mapsto E^*$.
Since $K^G_{alg}(\Theta)\to K^\G_{top}(\Theta)$ is an isomorphism, the map
 $$K^G_{alg}(V_{\succeq\Theta}) \to K^\G_{top}(V_{\succeq\Theta}) $$
is an isomorphism as well. That is so with or without precomposing with the duality, since the duality itself is an isomorphism.
\end{proof}
\begin{remark}\rm The map $(\iota^{alg}_\Theta)_*$ is injective. The sequence \eqref{algKex}  extended by 0 from the left remains exact.\end{remark}

\begin{remark}\rm Note that
$$K_\G ^*(V)\simeq \bigoplus_{[x]\in V/\G }\Rep(\G _x)\,.$$
This is an isomorphism  of abelian groups, not necessarily $\Rep(\G)$-modules.
A similar  situation appears for rational equivariant cohomology, \cite[Th.~1.7]{FW}.\end{remark}

More generally, suppose $V$ is smooth complex $\G $-variety, with a filtration $$\emptyset=U_0\subset U_1\subset U_2\subset \dots \subset U_N=V$$ consisting of $\G $-invariant open sets, such that $U_k\setminus U_{k-1}$ is a complex submanifold of $U_k$, which is $\G $--homotopy equivalent to an orbit $\G /\G _k$ for $k=1,2,\dots,N$. Then additively
$$K_\G ^*(V)\simeq \bigoplus_{k=1}^N\Rep(\G _k)\,.$$
The proof is by inductive application of the exact sequence
\[
\xymatrix{
0\ar[r]&
\Rep(\G _k)\ar[r]&
K^*_\G (U_k)\ar[r]&
K^*_\G (U_{k-1})\ar[r]&0\,.}
\]
The results of this paper, in particular Theorem \ref{thm:intchar} and Theorem \ref{thm:axiomatic}, with additional work might apply to this general situation.

\subsection{Lefschetz--Riemann--Roch}
The localization to the fixed points in both theories are related with push-forwards by the Lefschetz formula
\begin{theorem} \label{LRRgeneral} Let $f:X\to Y$ be a proper equivariant map $\T$-varieties. Then for all $\omega\in K^\T(X)$
\[\frac{f_*( \omega ){|Y^\T}}{\lambda_{-1}(\nu^*(Y^\T))}
=
f_*^\T\left(\frac{\omega{|X^\T}}{\lambda_{-1}(\nu^*(X^\T))}\right)\,,
\]
where $f_*^\T:K^\T(X^\T)\otimes_{\Rep(\T)}Frac(\Rep(\T))\to K^\T(Y^\T)\otimes_{\Rep(\T)}Frac(\Rep(\T)) $ is the scalar extension of the restricted map $f$.\end{theorem}
Here $\nu(X^\T)$ and $\nu(Y^\T)$ denote the normal bundles which might have varying dimensions over different components of the fixed point sets.
This form of  the localization  theorem K-theory can be found in \cite[Thm.~3.5]{Tho}, \cite[Thm. 5.11.7]{ChrissGinz} and originates from the results of \cite{BFQ} for finite groups. The topological counterpart is an extension of the well known Atiyah-Bott-Berline-Vergne localization formula in cohomology. If the fixed point sets are discrete, then the normal bundle becomes disjoint union of the tangent spaces and we arrive to the formulas like in Proposition \ref{LRR}.

\subsection{Completion and Chern character}

The Borel construction leads to a map
$$\pmb{\alpha}:K_\G ^*(X)\to K^*(E\G \times_\G  X)\,,$$
where the  K-theory of infinite CW-complexes are defined as homotopy classes of maps to the classifying space:
$$K^0(E\G \times_\G  X)=[E\G \times_\G  X\,,\,BU]\,,\qquad K^1(E\G \times_\G  X)=[E\G \times_\G  X\,,\,U]\,.$$ The map $\pmb{\alpha}$
extends to an isomorphism of the $I_\G $-adic completion
$$K^*_\G (X)_{I_\G }^\wedge\;\simeq\;K^*(E\G \times_\G  X)\,,$$
where  $I_\G \subset K^0_\G (\pt)=\Rep(\G )$ is the augmentation ideal
$$I_\G =
\ker(\dim:\Rep(\G )\to\Z)\,,$$ see \cite[Prop. 4.2]{AS}.
The Chern character of an equivariant vector bundle
is constructed by applying the nonequivariant Chern character to the associated vector bundle on $E\G \times_\G  X$. We obtain a map
$$ch_\G :K^*_\G (X)\to\hat H^*_\G (X;\Q)=\prod_{k=0}^\infty H^k_\G (X;\Q)$$
This Chern character may have a kernel which is not $\Z$-torsion.

\begin{example}\rm If $\G =S^1$, $\HH=\Z/(p)$, $X=\G /\HH$ then we have:
\begin{itemize}
\item The equivariant K-theory
$$K^*_\G (X)= \Rep(\HH)\simeq\Z[\Z/(p)]$$
is additively isomorphic to $\Z^p$.

\item
The K-theory of the Borel construction $K^*(E\G \times_\G X)\simeq K^*(B\HH)$ can be computed using the Atiyah-Segal completion theorem: $$K^*(B\Z/(p))\simeq \Rep(Z/(p))^\wedge_{I_{\Z/(p)}}\simeq \Z\oplus \Z^\wedge_p\,,$$
see \cite[\S8]{At}.

\item Rational equivariant cohomology is trivial in positive degrees
 $$\hat H^*(E\G \times_\G  X;\Q)=\hat H^*(B\Z/(p);\Q)\simeq \Q\,.$$
\end{itemize}
\end{example}
\bigskip

On the other  hand for connected groups we have inclusions:
If $\G =S^1$, then $K^*_\G (\pt)\simeq \Z[\xi,\xi^{-1}]$, $\hat H^*(B\G ;\Q)\simeq \Q[[x]]$ and the Chern character maps $\xi$ to $\exp(x)$. In this case the map $$K^*_\G (\pt)\longrightarrow \hat H^*(B\G ;\Q)$$ is injective.
In general, for a connected compact group $\G $ with the maximal torus $\Tc $ and the Weyl group $W$ we have a commutative diagram of injections
\[
\xymatrix{
\Rep(\Tc )^W\ar@{=}[r]&
\Rep(\G )\ar@{=}[r]&
K^*_\G (\pt)\ar@{^{(}->}[d]\ar@{^{(}->}[r]^{ch_\G \phantom{xx}}&
\hat H^*(B\G ;\Q)\ar@{=}[r]\ar@{^{(}->}[d]&
\hat H^*(B\Tc;\Q)^W\\
&\Rep(\Tc )\ar@{=}[r]&
K^*_\Tc (\pt)\ar@{^{(}->}[r]^{ch_\Tc \phantom{xx}}&
\hat H^*(B\Tc ;\Q)}
\]

\subsection{Bia{\l}ynicki-Birula-cells a.k.a.~attracting sets}
Assume that the variety $X$ is smooth and projective, and a torus $\T $ acts with finitely many fixed points. Then due to Bia{\l}ynicki-Birula decomposition (for a generic 1-parameter subgroup) $X$ is a sum of $\T $-equivariant algebraic cells. This is a filtrable system, but in general it is not a stratification. Then $K^\T_{alg}(X)\simeq K^\Tc_{top}(X)$ for the corresponding compact torus $\Tc$, and the equivariant K-theory has no $\Rep(\T )$-torsion. By the localization theorem the restriction map
$K_\T^*(X)\to K_\T^*(X^\T)$ is an injection.
Therefore for such torus actions we may check identities involving motivic Chern classes in the equivariant cohomology of the fixed points. We do not know to what extent Theorem \ref{thm:axiomatic} holds in that generality.

\end{document}